\documentclass[10pt,letterpaper]{article}
\usepackage[utf8]{inputenc}

\usepackage{kashsty}
\usepackage{kashdef}
\usepackage{kashThm}

\usepackage{algorithm}
\usepackage{algorithmic}

\usepackage{bbm}
\usepackage[round,authoryear]{natbib}

\usepackage[margin=1.37in]{geometry}

\title{
Asymptotic symmetry and group invariance 
for randomization
}
\author{Adam B Kashlak\\
Mathematical \& Statistical Sciences\\
University of Alberta\\
Edmonton, Canada,  T6G 2G1}

\begin{document}

\maketitle

\begin{abstract}
  Symmetry is a cornerstone of much of mathematics,
  and many probability distributions possess 
  symmetries characterized by their invariance
  to a collection of group actions.
  Thus, many mathematical and statistical methods
  rely on such symmetry holding and ostensibly
  fail if symmetry is broken.
  This work considers under what conditions a 
  sequence of probability measures asymptotically
  gains such symmetry or invariance to a collection
  of group actions.  Considering the many
  symmetries of the Gaussian distribution, this
  work effectively proposes a non-parametric 
  type of central limit theorem.
  That is, a Lipschitz function of a high dimensional
  random vector will be asymptotically invariant 
  to the actions of certain compact topological 
  groups.
  Applications of this include a partial law
  of the iterated logarithm for uniformly random
  points in an $\ell_p^n$-ball and an asymptotic 
  equivalence between classical parametric statistical
  tests and their randomization counterparts even
  when invariance assumptions are violated.
\end{abstract}

\section{Introduction}

The central limit theorem stands as one of the most
fundamental results in probability theory.  In its
essence, sequences of measures properly normalized
asymptotically approach the Gaussian, 
and the standard Gaussian 
measure in, say, $\real^n$ satisfies many symmetries:
reflective, permutative, rotational, etc. 
This elicits a question: under what conditions
will a sequence a measures asymptotically achieve
certain symmetries?

In mathematics, the symmetry of an object is encoded by 
said object's invariance to a collection of group
actions.  For example, the standard Gaussian measure
on $\real^n$ is invariant to any element of the 
orthogonal group $O(n)$ comprised of all
$n\times n$ matrices 
$M:\real^n\rightarrow\real^n$ such that 
$\TT{M}M = M\TT{M} = I$.
Such symmetry is often leveraged in 
mathematical proofs,
statistical estimation and testing procedures,
and elsewhere.  As a consequence, 
the lack of such group 
invariance can quickly invalidate a result.
Thus, this work seeks to determine conditions
under which group invariance is achieved in 
an asymptotic sense.  That is, for a sequence
of random vectors $X^{(n)}\in\real^n$ and a 
sequence of functions $T_n:\real^n\rightarrow\real$
and mappings $\pi_n:\real^n\rightarrow\real^n$,
how close are the random variables 
$T_n( X^{(n)} )$ and $T_n( \pi_nX^{(n)} )$ as 
$n$ grows to infinity?

\begin{example}
  Let $X^{(n)} = (X_1,\ldots,X_n)$ have \iid 
  mean zero unit variance entries and 
  $T_n(X^{(n)}) = n^{-1/2}\sum_{i=1}^n X_i$.
  If $\pi_n$ is a permutation on $n$ elements,
  then obviously
  $$
    T_n(\pi_nX^{(n)}) = n^{-1/2}\sum_{i=1}^n X_{\pi(i)}
    = n^{-1/2}\sum_{i=1}^n X_{i} = T_n(X^{(n)}).
  $$
  If instead $\pi_n$ is a reflection,
  $\pi_n(X_1,\ldots,X_n) = 
  (\veps_1X_1,\ldots,\veps_nX_n)$ for some
  $\veps_i\in\{-1,1\}$, then the distributions
  of
  $T_n(\pi_nX^{(n)})$ and $T_n(X^{(n)})$ only 
  coincide if $X_1 = -X_1$ in distribution.
  Nevertheless, if such reflective symmetry
  fails to hold, $T_n(\pi_nX^{(n)})$ and $T_n(X^{(n)})$
  still asymptotically converge in distribution
  to the same $\distNormal{0}{1}$ distribution.
\end{example}

The main results of this work are contained in 
Section~\ref{sec:RandTest}.  
Theorem~\ref{thm:condProb} 
and Corollary~\ref{cor:condProb}
demonstrate the almost sure equality 
between a probability measure on some 
Hilbert space and Haar measure on 
some subset of a compact topological group $G$.
These results are highly reminiscent of similar
results from \cite{LEHMANN2006} and
\cite{HEMERIK2018} and others.
The novel theorem in this work is
Theorem~\ref{thm:asympInvariance} with its 
paired Corollary~\ref{cor:asympSize}. 
These extend the previous results by considering
conditions on the random vector $X^{(n)}$,
the function $T_n$, and the group of symmetries $G$
under which
almost sure convergence and convergence in mean
between the two probabilities occur.

The first application comes in 
Section~\ref{sec:lpballs}, which considers
the group of rotations $SO(n)$ applied to 
uniform measures on $\ell_p^n$-balls.
The result is a partial law of the iterated
logarithm (LIL), which is an almost sure upper bound
of the form
$$
  \limsup_{n\rightarrow\infty}
  \frac{
    \abs*{\sum_{i=1}^{n} X_i}
  }{Kn^{1/2-1/p}\sqrt{ \log\log n }}
  \le 1, ~~\text{a.s.}
$$
for some constant $K>0$
and $(X_1,\ldots,X_n)$ a uniformly random point
within $\ell_p^{n}$.  Of note, these
$X_i$ are not independent unlike the usual
LIL setting.  Showing the above 
equals $1$ almost surely is left to future work
as proving the lower bound typically relies
on the second Borel-Cantelli lemma. This, in 
turn, relies on independence of events, which 
is not the case in this problem.
Furthermore, the value of $K$ resulting from 
certain concentration inequalities on $SO(n)$
is suboptimal.
For independent random variables, there is a 
long history of classical LIL results that
are detailed in Chapter 5 of \cite{STOUT1974}
and Chapter 8 of \cite{LEDOUXTALAGRAND1991}.

The second application comes in 
Section~\ref{sec:specificTests}, which considers
classical one-sample and two-sample statistical 
hypothesis testing.  Nonparametric randomization
tests can be performed using the group of 
reflections and permutations, respectively,
but only when the data distribution is 
invariant to such group actions.  
Section~\ref{sec:specificTests} first applies 
Theorem~\ref{thm:asympInvariance} to these
testing problem to show when, for large $n$,
the distribution of the test statistic is nearly
invariant to the corresponding group actions.
It secondly proves some quantitative versions
of Theorem~\ref{thm:asympInvariance} in these
specific settings; see Theorems~\ref{thm:oneSample}
and~\ref{thm:twoSample}.  These are in the style
of Berry-Esseen bounds for the central limit theorem
\citep{FELLER2008B}.

Randomization testing in statistics arguably
dates back to Sir Ronald Fisher himself
with the formulation of the permutation test, 
but it has gained more popularity 
with the advent of modern computers
\citep{PESARIN2010,GOOD2013}.
There has been much modern work on the usage
of random permutations for hypothesis testing
\citep{HEMERIK2018,HEMERIK2018false,HEMERIK2019,KASHLAK_YUAN_ABELECT,KASHLAK_KHINTCHINE2020,KIM2022,VESELY2023}.  There has also been some very recent work on 
subgroup and subset selection for more efficient 
permutation testing 
\citep{BARBER2022,KONING2023,KONING_HEMERIK_2023}.
Beyond permutations, a few works have considered
hypothesis testing with random rotations
\citep{SOLARI2014}, the wild bootstrap effectively
uses random reflections to construct confidence
regions \citep{KASHLAK_BURAK_WILDBS}, and
other recent works consider 
general group symmetry \citep{CHIU2023}.

To correctly apply a two-sample t-test for testing
for equality of population means, the populations 
should be Gaussian distributed with homogeneous 
variances.  This ensures that the classic student's
t-test is exact.
That is, the desired type-I error rate is achieved
for all finite sample sizes.
However in practice, the t-test is applied to 
a vast number of scenarios without producing erroneous 
results.  This is in part thanks to the central limit
theorem and asymptotic normality.  
If the sample size is large enough, the 
sample means as they appear in the t-statistic will
be close enough to Gaussian to ensure some trust
in the final result.  The notion of `close' to 
Gaussian can be made more precise through results
like the aforementioned 
Berry-Esseen bounds.

In contrast, to correctly apply a 
two-sample permutation test for testing for 
equality of population means, the joint distribution
of the sample should be exchangeable, i.e.~invariant
to permutations, under the null hypothesis.  This
is a slightly weaker condition than assuming an 
iid sample, and thus reasonable in many practical 
settings.  However, as an example, 
heterogeneous population variances will violate 
the exchangeability assumption.  A motivating goal 
of this work is to determine if the permutation
test is robust to slight deviations from exchangeability
much like the t-test's robustness to slight deviations
from Gaussianity.

\section{Asymptotic Invariance}
\label{sec:RandTest}

Let $(\Omega,\mathcal{F},P)$ be a probability space,
$(H,\mathcal{H})$ be an Hilbert space
space equipped with $\mathcal{H}$, the
Borel $\sigma$-field generated
by the norm topology,
and let
$(\real,\mathcal{B})$ denote the real line with the 
standard Borel $\sigma$-field.  
Let $X:\Omega\rightarrow H$
be an $H$-valued random variable, and let $T:H\rightarrow\real$
be a measurable function.
Let $G$ be a compact topological group,
which implies that the multiplication, $(g,h)\rightarrow gh$, 
and inversion, $g\rightarrow g^{-1}$, operations
are continuous.
The group $G$ comes equipped with normalized 
Haar measure $\rho$ being the unique left-invariant 
measure on the 
Borel sets derived from the topology on $G$;
that is, $\rho(E) = \rho(gE)$ for any $g\in G$
and any measurable 
subset $E$ of $G$.  Existence and uniqueness of $\rho$
is discussed in chapter 2 of \cite{HOFMANNMORRIS}
among other sources.
Elements of $G$ are said to act on $H$,
i.e.~for $\mathcal{L}(H)$, the space of $H$-endomorphisms
with the strong operator topology, 
there exists a 
representation $\pi:G\rightarrow\mathcal{L}(H)$
with $\pi(g) = \pi_g$ such that  $\pi_g:H\rightarrow H$
and $\pi(gh)=\pi(g)\pi(h)$ for any $g,h\in G$.
The Hilbert space $H$ is called an Hilbert $G$-Module;
see ``Weyl’s Trick'' in Theorem~2.10 of \cite{HOFMANNMORRIS}
for the existence of an inner product that makes all
$\pi_g$ unitary.
A set $S\in\mathcal{H}$ is said to be $G$-invariant if
$S = \pi_gS$ for all $g\in G$.
Then, it can be readily checked that
the collection $\mathcal{S}$ of $G$-invariant sets
forms a $\sigma$-field and $\mathcal{S}\subseteq\mathcal{H}$.

Two conditions on the above will be considered in this work:
\begin{itemize}
  \item[\bf C1] The measure on $H$ induced by $X(\omega)$ is
  $G$-invariant, i.e.,
  for all $g\in G$ and $A\in\mathcal{H}$,
  $\prob{ X\in A } = \prob{\pi_gX\in A}$
  where $\prob{ X\in A } := 
  P(\{\omega\in\Omega\,:\,X(\omega)\in A\})$.
  \item[\bf C2] The measure on $\real$ 
  induced by the mapping $T(X(\omega))$ is 
  $G$-invariant, i.e.,
  for all $g\in G$ and $B\in\mathcal{B}$,
  $\prob{ T(X)\in B } = \prob{T(\pi_gX)\in B}$
  where $\prob{ T(X)\in B } := 
  P(\{\omega\in\Omega\,:\,T(X(\omega))\in B\})$.
\end{itemize}
Condition C2 implies that 
$\prob{ X\in T^{-1}(B) } = \prob{ X \in \pi_gT^{-1}(B)}$ 
for all $g\in G$
or, that is, that condition C1 holds restricted to 
$\mathcal{T}=\sigma(T)$
rather than on all of $\mathcal{H}$ where 
$\sigma(T)$ is the smallest $\sigma$-field on $H$
such that $T$ is measurable. 

\begin{example}
  As a simple example, let $T:\real^2\rightarrow\real$
  be $T( (x,y) ) = x-y$ and $G = \mathbb{S}_2$, 
  the symmetric group on two elements.  
  Then, $B\in\mathcal{S}$ if whenever $(x_0,y_0)\in B$, 
  then $(y_0,x_0)\in B$; i.e. reflections across the diagonal.
  And $A\in\mathcal{T}$ if whenever $(x_0,y_0)\in A$,
  then $(x_0+r,y_0+r)\in A$ for any $r\in\real$; i.e.
  diagonal strips.  This is displayed in 
  Figure~\ref{fig:egST}.
\end{example}

\begin{figure}
  \centering
  \includegraphics[width=0.31\textwidth]{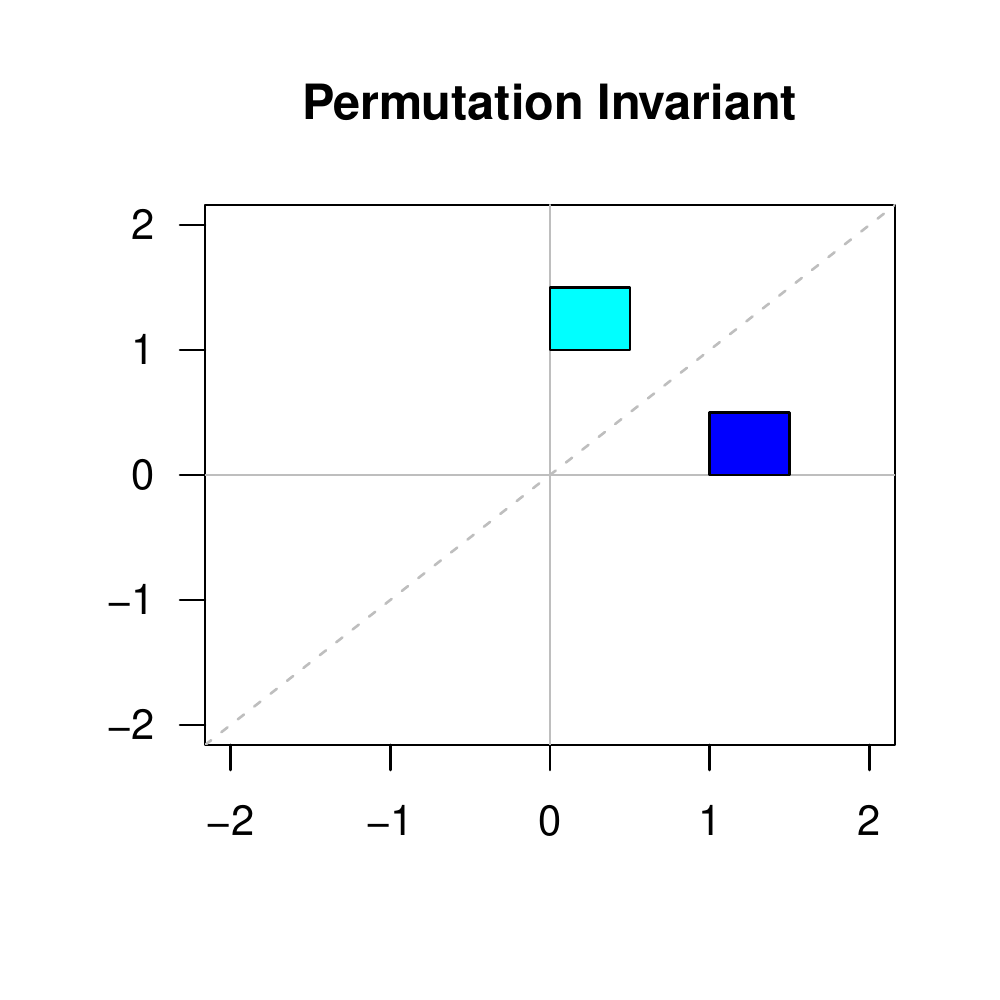}
  \includegraphics[width=0.31\textwidth]{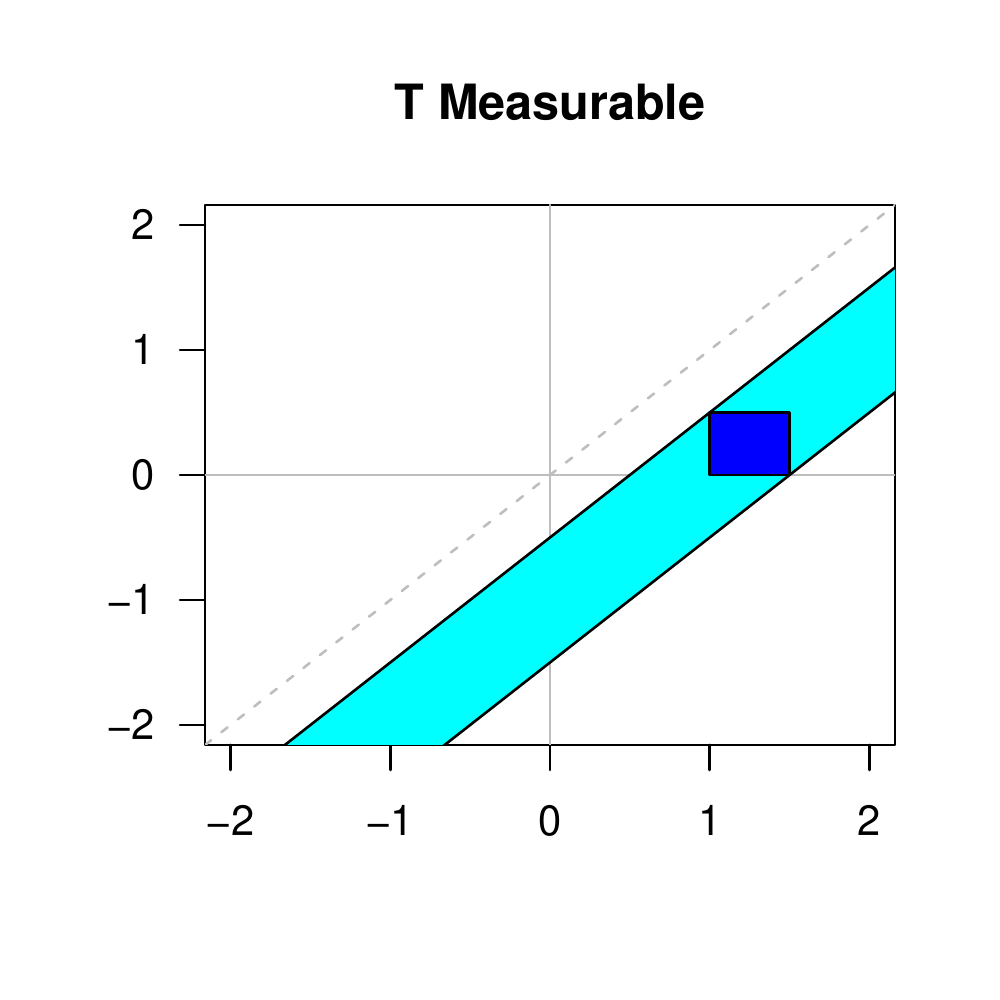}
  \includegraphics[width=0.31\textwidth]{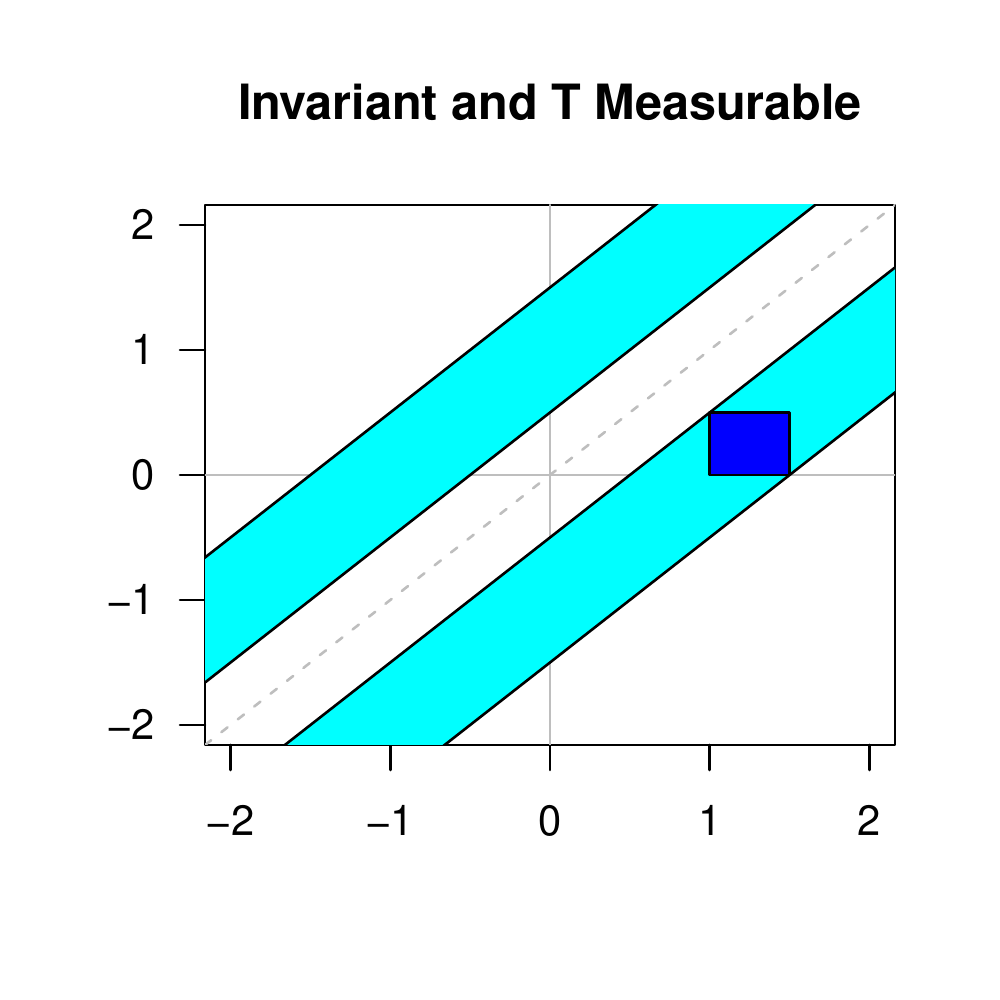}
  \caption{
    \label{fig:egST}
    Starting with the blue square, the cyan set includes
    all the necessary points to make the set $\mathbb{S}_2$
    invariant, $\mathcal{T}$-measurable, or in
    $\mathcal{S}\cap\mathcal{T}$.
  }
\end{figure}

It is clear that condition C1 implies condition C2.  
Furthermore, C1 has been 
studied in \cite{LEHMANN2006} and others referred to as 
the \textit{Radomization Hypothesis} or
\textit{Total Radomization Hypothesis}. However, this
is typically unnecessarily strong for hypothesis testing.
An example of this is the two sample t-test discussed below
in Section~\ref{sec:TwoTest} where C2 is achieved either when
the two populations have homogeneous variances or when the 
sample sizes coincide.  C2 is also achieved if 
$T$ is a $G$-invariant-mapping, i.e.~$T(x)=T(\pi_gx)$
for all $g\in G$ and $x\in H$.  However, in this case, 
the utility of $T$ is lost as the invariance
or lack thereof with respect to specific probability measures
is of primary interest.
For a simple example that is explored 
more in Section~\ref{sec:specificTests}, 
the sample mean is invariant to 
permutations, but the difference between two 
sample means may not be invariant; i.e. 
addition is commutative, and subtraction is not.

Given a fixed $x\in H$ and some $\alpha\in(0,1)$, 
the randomization 
threshold $t_\alpha(x)\in\real$ can be defined as
$$
  t_\alpha(x) = \inf\left\{
  {t\in\real} \,:\,
  \rho(\{
    g\in G\,:\, T(\pi_gx)>t
  \})\le\alpha
  \right\}
$$
Therefore,
$
\rho(\{
    g\in G\,:\, T(\pi_gX)>t_\alpha(X)
  \})\le\alpha
$
$P$-almost surely.  This somewhat simple fact is critical for
the following Corollaries~\ref{cor:condProb}
and~\ref{cor:asympSize}.

The following is a slightly modified version of Theorem
15.2.2 from \cite{LEHMANN2006} with some similar 
results appearing in \cite{HEMERIK2018}.

\begin{theorem}
  \label{thm:condProb}
  Under condition C2, 
  $$
    \prob{
      T(X) \in B \,\mid\, \mathcal{S}\cap\mathcal{T}
    } = 
    \rho\left(\{
      g\in G \,:\,
      T(\pi_gX) \in B
    \}\right)
  $$
  $P$-almost surely 
  for any $B\in\mathcal{B}$.
\end{theorem}

\begin{proof}
  For any $G$-invariant set $S\in\mathcal{S}\cap\mathcal{T}$,
  there exists an $A\in\mathcal{B}$ such that $S = T^{-1}(A)$.
  Thus, denoting $\xv_X[\cdot]$ the expected value with respect to
  $X$,
  \begin{align*}
    \xv_X&\left[
      \rho\left(\{
      g\in G \,:\,
      T(\pi_gX) \in B
    \}\right)\indc{S}
    \right]\\
    &= 
    \int_S
    \int_{G}
    \Indc{
      T(\pi_gx)\in B
    }
    d\rho(g)
    dP(x)\\
    &= 
    \int_{G}
    \int_S
    \Indc{
      T(\pi_gx)\in B
    }
    dP(x)
    d\rho(g)\\
    &= 
    \int_{G}
    \prob{
      T(\pi_gX)\in B,
      X\in S
    }
    d\rho(g)\\
    &= 
    \int_{G}
    \prob{
      T(\pi_gX)\in B,
      \pi_gX\in S
    }
    d\rho(g)\\
    &= 
    \int_{G}
    \prob{
      \pi_gX\in T^{-1}(B)\cap T^{-1}(A)
    }
    d\rho(g)\\
    &= 
    \int_{G}
    \prob{
      X\in T^{-1}(B\cap A)
    }
    d\rho(g)= \prob{
      T(X)\in B,
      X\in S
    }.
  \end{align*}
  Thus, by uniqueness, the conditional probability
  $\prob{
      T(X) \in B \,\mid\, \mathcal{S}\cap\mathcal{T}
    }$
    coincides with the random measure
    $
    B \rightarrow \rho\left(\{
      g\in G \,:\,
      T(\pi_gX) \in B
    \}\right)
    $.
\end{proof}

\begin{corollary}
  \label{cor:condProb}
  Under condition C2, 
  $
    \prob{
      T(X) > t_\alpha(X)
    } \le \alpha.
  $
\end{corollary}
\begin{proof}
  Let $R = \{x\in H\,:\, T(x)>t_\alpha(x)\}$.
  Then, $R\in\mathcal{H}$ and from Theorem~\ref{thm:condProb}
  \begin{align*}
    \prob{ T(X) > t_\alpha(X) } 
    &=
    \xv\left[
      \prob{
        T(X) > t_\alpha(X) \,\mid\, \mathcal{S}\cap\mathcal{T}
      }
    \right]\\
    &= 
    \xv\left[
      \rho\left(\{
      g\in G \,:\,
      T(\pi_gX) > t_\alpha(X)
    \}\right)
    \right] \le \alpha
  \end{align*}
  as almost sure equality implies equality in mean.
\end{proof}

The validity of the above result hinges on condition 
C2.  Upon removal of that condition, 
almost sure equality is lost, but can still be achieved 
asymptotically.  In what follows, let $H = \real^\infty$.
For all $n\in\natural$, let 
$\mathcal{S}_n$ be the $\sigma$-field of 
$G_n$-invariant sets
in $\real^n$. 
Furthermore, let $T_n:\real^n\rightarrow\real$ be 
$c_n$-Lipschitz and 
$\mathcal{T}_n$ be the smallest $\sigma$-field
on $\real^n$ such that $T_n$ is measurable.
The choice of the sequence Lipschitz constants
has intriguing implications that are discussed
post-theorem in Section~\ref{sec:lipSelection}.

Functions $T_n$, groups $G_n$, and sets $S_n$
can be extended to $\real^\infty$.
A set $S_n\in\mathcal{S}_n$ can be written as 
$\{S_n\otimes\real\otimes\ldots\}\subset \real^\infty$
which is invariant to elements of $G_n$ acting on the
first $n$ coordinates and fixing the rest.
Let $G := \bigcup_{n\ge1}G_n$ which 
consists of all group actions from $G_n$ that
only modify the first $n$ entries of 
$x\in\real^\infty$ for all $n\in\natural$; e.g. $G$
may consist of all permutations that only permute a 
finite number of elements as arises in the Hewitt-Savage
zero-one law \citep{RAO1974,DUDLEY2002}.
Tychnoff's theorem ensures 
compactness is maintained in the limit with respect
to the product topology; of note, arbitrary products 
of compact groups are compact as are subgroups of 
such \citep[Proposition 1.14]{HOFMANNMORRIS}.
A tail set $E\subset\real^\infty$, 
  as defined by \cite{OXTOBY2013}, 
  is such that if $x\in E$ and if $y$ differs from $x$
  in only a finite number of coordinates, then $y\in E$.
  Since any $g\in G$ necessarily modifies only a finite
  number of coordinates, $x\in E$ implies $\pi_gx\in E$
  and thus tail sets are $G$-invariant.
Furthermore, $\mathcal{S}$ is the $\sigma$-field 
on $\real^\infty$ of $G_n$-invariant sets for 
all $n\in\natural$, and $T:\real^\infty\rightarrow\real$
is defined as $T:=\lim_{n\rightarrow\infty}T_n$ where
$T_n$ can be defined on $\real^\infty$  
by projecting $x\in\real^\infty$ onto the first $n$
coordinates.  A simple example is the sample mean
$T_n(x) = n^{-1}\sum_{i=1}^nx_i$, which will be 
discussed in Section~\ref{sec:specificTests}.

In what follows, the notion of a \textit{L\'evy family}
is required \citep{GROMOVMILMAN1983,LEDOUX2001}.
Let $(M^{(n)},d^{(n)},\mu^{(n)})$ be a family of 
metric measure spaces for $n\ge1$. 
The open neighbourhood of a set $A\subset M^{(n)}$ 
for some $t>0$ is 
$A_t = \{ x\in M^{(n)} \,:\, d^{(n)}(x,A)<t \}$.
This collection of metric measure spaces
is said to be a \textit{normal L\'evy family} if
$$
  \sup_{A\subset M^{(n)}}\left\{
    1 - \mu^{(n)}(A_t) \,:\,
    \mu^{(n)}(A) \ge 1/2
  \right\}
  \le 
  K\ee^{-knt^2}
$$
for some constants $K,k>0$.  From the previous paragraph, 
$M^{(n)} = G_n$ treated as a subgroup of $G$ that acts as
the identity on all coordinates $i>n$.  The measure
$\mu^{(n)}=\rho_n$ will be Haar measure for $G_n$.  The 
main results below require the family $(G_n,d_n,\rho_n)$
be a normal L\'evy family.  This, of course, covers a 
wide variety of groups.  Most notably, the 
classical compact groups ${SO}(n)$, 
${SU}(n)$, and ${S}{p}(2n)$ 
with the Hilbert-Schmidt metric 
satisfy this requirement \citep[Chapter 5]{MECKES2019}.
Furthermore, any sequence of topological 
groups corresponding to compact connected smooth
Riemannian manifolds with geodesic distance
and strictly positive Ricci curvature
embedded in $\real^n$
\citep{GROMOV1999,LEDOUX2001,MILMANSCHECHTMAN2009}.  
For discrete groups,
the reflection group (Section~\ref{sec:OneTest})
with the Hamming metric
is normal L\'evy following from Hoeffding's 
inequality.  The symmetric group (Section~\ref{sec:TwoTest})
and many other compact groups
are also normal L\'evy; see Corollary 4.3 and 
Theorem 4.4 in \cite{LEDOUX2001}.

The following theorem shows that 
Theorem~\ref{thm:condProb} can hold in an asymptotic 
sense.  It is proven via the three subsequent lemmas 
below.  Lastly, an asymptotic analogue of Corollary~\ref{cor:condProb}
is stated and proved below in Corollary~\ref{cor:asympSize}.

\begin{theorem}
  \label{thm:asympInvariance}
  Let $X\in\real^\infty$ and $X^{(n)}\in\real^n$ be $X$
  projected onto its first $n$ coordinates.
  Let $T_n$ be $c_n$-Lipschitz such that for some $p\ge1$,
  $\xv\norm{X^{(n)}}^p<\infty$ for all $n$ and 
  $\sum_{n=1}^\infty c_n^p<\infty$.
  Furthermore, let $n^{-1/2}c_n\norm{X^{(n)}}\convas0$.
  Lastly, let the collection of $G_n$ be a normal
  L\'evy family.
  Then,
  $$
  \abs*{
    \rho\left(\{
      g\in G_n \,:\,
      T_n(\pi_gX^{(n)}) > t
    \}\right)
    - \prob{
      T_n(X^{(n)}) > t  \,|\,
      \mathcal{S}_n\cap\mathcal{T}_n
    }
    }
    \rightarrow 0
  $$
  $P$-almost surely and in $L^1$
  as $n\rightarrow\infty$.
\end{theorem}

\begin{proof}
  Decomposing the difference gives
  \begin{align*}
    &\abs*{
    \rho\left(\{
      g\in G_n \,:\,
      T_n(\pi_gX^{(n)}) > t
    \}\right)
    - \prob{
      T_n(X^{(n)}) > t  \,|\,
      \mathcal{S}_n\cap\mathcal{T}_n
    }
    }\\
    &~~~\le 
    \abs*{
    \rho\left(\{
      g\in G_n \,:\,
      T_n(\pi_gX^{(n)}) > t
    \}\right)
    - 
    \rho\left(\{
      g\in G \,:\,
      T(\pi_gX) > t
    \}\right)
    }&&\mathrm{(I)}\\
    &~~~~+ 
    \abs*{
    \rho\left(\{
      g\in G \,:\,
      T(\pi_gX) > t
    \}\right)
    -
    \prob{
      T(X) > t  \,|\,
      \mathcal{S}\cap\mathcal{T}
    }
    }&&\mathrm{(II)}\\
    &~~~~+ 
    \abs*{
    \prob{
      T_n(X^{(n)}) > t  \,|\,
      \mathcal{S}_n\cap\mathcal{T}_n
    }
    -
    \prob{
      T(X) > t  \,|\,
      \mathcal{S}\cap\mathcal{T}
    }
    }&&\mathrm{(III)}
  \end{align*}
  
  The three pieces will be dealt with by the subsequent lemmas.
  Part (I) is handled
  by concentration of measure for compact topological 
  groups in Lemma~\ref{lem:groupAverageConvergence}.
  Part (II) follows from Lemma~\ref{lem:asympInvarT},
  which implies
  condition C2 in the limiting case. 
  Thus, by Theorem~\ref{thm:condProb} 
  $\rho\left(\{
      g\in G \,:\,
      T(\pi_gX) > t
    \}\right)
    =
    \prob{
      T(X) > t  \,|\,
      \mathcal{S}\cap\mathcal{T}
    }$
  $P$-almost surely and thus equal in mean as well. 
  Lastly, part (III) is handled by martingale 
  convergence in Lemma~\ref{lem:martConverge}.
\end{proof}

\begin{lemma}
  \label{lem:asympInvarT}
  For $X\in\real^\infty$, 
  let $X^{(n)}$ be $X$ projected 
  onto $\real^n$ by taking the first $n$ coordinates.
  Let $T_n$ be $c_n$-Lipschitz with $c_n\rightarrow0$.
  Assuming $\xv\norm{X^{(n)}}<\infty$ for each $n\in\natural$,
  the function $T_{n+m}(X)$ is asymptotically 
  $G_n$-invariant in mean
  for any fixed $n\in\natural$, 
  i.e. 
  $\xv\abs{T_{n+m}(\pi_gX)-T_{n+m}(X)}\rightarrow0$ for any
   fixed $g\in G_n$ and $n\in\natural$ as $m\rightarrow\infty$.
  Furthermore, if for some $p\ge1$,
  $\xv\norm{X^{(n)}}^p<\infty$ and 
  $\sum_{n=1}^\infty c_n^p<\infty$
  then $T_{n+m}(X)$ is asymptotically $G_n$-invariant  
  $P$-almost surely.
\end{lemma}
\begin{proof}
  For a fixed $n$ and any $g\in G_n$ and any $m\in\natural$, 
  there exists a unitary representation 
  $\pi_g:\real^{n+m}\rightarrow\real^{n+m}$ that 
  fixes the final $m$ coordinates.  That is,
  $$
    \pi_g = \begin{pmatrix}
      \pi_g^{(n)} & 0\\
      0 & I^{(m)}
    \end{pmatrix}.
  $$
  Writing 
  $X^{(n+m)} = (X^{(n)},X^{(m)})$ and 
  $\pi_g = ( \pi_g^{(n)}, I^{(m)})$
  where $I^{(m)}$ is the identity mapping on $\real^m$
  results in 
  $\pi_g X^{(n+m)} = (\pi_g^{(n)}X^{(n)},X^{(m)})$.
  Thus, for any fixed $n$
  $$
    \abs{T_{n+m}(\pi_g X) - T_{n+m}(X)}
    \le c_{n+m} \norm{\pi_gX - X} 
    \le 2 c_{n+m} \norm{X^{(n)}},
  $$
  Taking the expectation and $m\rightarrow\infty$
  proves asymptotic invariance in mean.

  Secondly, by Markov's inequality, for any $t>0$
  \begin{align*}
    \prob{
      \abs{T_{n+m}(\pi_g X) - T_{n+m}(X)} > t
    }
    &\le 
    \prob{
      2 c_{n+m} \norm{X^{(n)}}> t
    }\\
    &\le 
    2^pc_{n+m}^p{\xv \norm{X^{(n)}}^p}t^{-p}.
  \end{align*}
  Hence, almost sure convergence follows from the 
  Borel-Cantelli lemma and the assumptions
  that $\sum_{n=1}^\infty c_n^p<\infty$.
\end{proof}


\begin{lemma}
  \label{lem:martConverge}
  The sequence of conditional probabilities 
  $
    \prob{ T(X^{(n)})>t \,|\, \mathcal{S}_n\cap\mathcal{T}_n }
  $
  converges to 
  $    
    \prob{ T(X)>t \,|\, \mathcal{S}\cap\mathcal{T} }
  $
  almost surely and in $L^1$
  as $n\rightarrow\infty$.
\end{lemma}
\begin{proof}
  Let 
  $Z_{n,m} := \xv[\indc{ T(X^{(n)})>t }|\mathcal{S}_m\cap\mathcal{T}_m]$,
  $Z_{n,\infty} := \xv[\indc{ T(X^{(n)})>t }|\mathcal{S}\cap\mathcal{T}]$, 
  $Z_{\infty,m} := \xv[\indc{ T(X)>t }|\mathcal{S}_m\cap\mathcal{T}_m]$, and
  $Z := \xv[\indc{ T(X)>t }|\mathcal{S}\cap\mathcal{T}]$.
  For $S\in \mathcal{S}_m$, $S = \{S^{(m)}\times\real\times\ldots\}$ 
  and similarly for sets in $\mathcal{T}_m$.  Hence,
  $\mathcal{S}_m\cap\mathcal{T}_m \subset
  \mathcal{S}_{m+1}\cap\mathcal{T}_{m+1}$, and thus
  the sequence $Z_{n,m}\rightarrow Z_n$ almost
  surely and in $L^1$ for any fixed $n$ as a 
  consequence of Levy's
  Upward Lemma; see \cite{ROGERSWILLIAMS1} section II.50.
  The same holds for $Z_{\infty,m}\rightarrow Z$.
  But furthermore, $T(X^{(n)})$ is, of course, 
  $\mathcal{T}_n$-measurable. Hence, for any fixed $n$ and
  all $m_1,m_2\ge n$, $Z_{n,m_1}=Z_{n,m_2}$
  almost surely.
  Hence, 
  \begin{equation}
    \label{eqn:ASEquality}
    \xv[\indc{ T(X^{(n)})>t }|\mathcal{S}_n\cap\mathcal{T}_n]
    =
    \xv[\indc{ T(X^{(n)})>t }|\mathcal{S}\cap\mathcal{T}]
  \end{equation}
  almost surely, and by the conditional dominated convergence
  theorem (\cite{ROGERSWILLIAMS1} section II.41),
  \begin{align*}
    \xv[\indc{ T(X^{(n)})>t }|\mathcal{S}_m\cap\mathcal{T}_m]
    &\convas 
    \xv[\indc{ T(X)>t }|\mathcal{S}_m\cap\mathcal{T}_m], \text{ and}\\
    \xv[\indc{ T(X^{(n)})>t }|\mathcal{S}\cap\mathcal{T}]
    &\convas
    \xv[\indc{ T(X)>t }|\mathcal{S}\cap\mathcal{T}]
  \end{align*}
  as $n\rightarrow\infty$.
  
  As a consequence of Equation~\ref{eqn:ASEquality}, the sequence
  $\{Z_{n,n}\}_{n=1}^\infty$ is almost surely equal to 
  $\{Z_{n,n+k}\}_{n=1}^\infty$ for any $k\in\natural$.
  As equality holds for all $k$, 
  $\{Z_{n,n}\}_{n=1}^\infty$ is almost surely equal to 
  $\{Z_{n,\infty}\}_{n=1}^\infty$.  As noted above,
  $Z_{n,\infty} \convas Z$ via dominated convergence.  Hence,
  $Z_{n,n}$ converges almost surely as well.
  
  Lastly, a classic theorem of Doob \citep[section II.44]{ROGERSWILLIAMS1}
  implies that the $Z_{n,m}$ are uniformly integrable.
  Hence, uniform integrability and convergence almost surely
  (in probability) implies convergence in $L^1$ by 
  Theorem 21.2 in chapter II of \cite{ROGERSWILLIAMS1}.
\end{proof}

\begin{lemma}
  \label{lem:groupAverageConvergence}
  For each $n$, let $G_n$ be a normal L\'evy family
  with respect to normalized Haar measure.
  Let $X\in\real^\infty$ be a random variable
  with projection $X^{(n)}\in\real^n$ onto the 
  first $n$ coordinates, and let 
  $T_n$ be $c_n$-Lipschitz. 
  If $n^{-1/2}c_n\norm{X^{(n)}}\convas0$
  then
  $$
    \abs*{
    \rho(\{
      g\in G_n \,:\,
      T_n(\pi_gX^{(n)}) > t
    \}) - \indc{T(X) > t}
    } \rightarrow 0 
  $$
  $P$-almost surely and in $L^1$ as $n\rightarrow\infty$.
\end{lemma}
\begin{proof}
  Let $f_{x,n}:G_n\rightarrow\real$ be defined as 
  $f_{x,n}(g) = T_n(\pi_gx)$.  Then, for any 
  $g,h\in G_n$ with unitary representations 
  $\pi_g,\pi_h\in \mathcal{L}(\real^n)$,
  \begin{align*}
    \abs{f_{x,n}(g)-f_{x,n}(h)} &=
    \abs{T(\pi_gx)-T(\pi_hx)} \\
    &\le 
    c_n\norm{\pi_gx-\pi_hx}\\
    &\le 
    c_n\norm{x}\norm{\pi_g-\pi_h}_{\mathcal{L}(\real^n)}.
  \end{align*}
  Thus, $f_{x,n}$ is $c_n\norm{x}$-Lipschitz on 
  $\mathcal{L}(\real^n)$
  with respect to the operator norm.  
  As a consequence of the $G_n$ forming a normal
  L\'evy family,
  there exists fixed constants $K,k>0$ 
  such that for all 
  $t\ge0$ and $n\ge1$
  $$
    \rho\left(
      \abs{f_{x,n}(g) - \int f_{x,n}(g)d\rho(g)}>t
    \right)
    \le K\exp\left(
      -\frac{knt^2}{2c_n^2\norm{x}^2}
    \right).
  $$
  Consequently,
  \begin{equation}
  \label{eqn:concIneqManifolds}
  \rho(\{
      g\in G_n \,:\,
      T_n(\pi_gX^{(n)}) > t
    \}) \le
    K\exp\left(
      -\frac{nk}{2c_n^2\norm{X^{(n)}}^2}
      \left(t-\int f_{X^{(n)}}(g)d\rho(g)\right)_+^2
    \right).
  \end{equation}
  By Jensen's inequality, the right hand side of 
  inequality~\ref{eqn:concIneqManifolds} is
  bounded above by
  $$
    K\int\exp\left(
      -\frac{nk}{2c_n^2\norm{X^{(n)}}^2}
      \left(t-f_{X^{(n)}}(g)\right)_+^2
    \right)d\rho(g).
  $$
  For any fixed $\omega\in\Omega$ and any $\veps>0$, let 
  $t_{\omega,\veps} = T(X(\omega))+\veps$.
  Then, by dominated convergence, the 
  assumption that 
  $n^{-1/2}c_n\norm{X^{(n)}}\convas0$, and 
  Lemma~\ref{lem:asympInvarT},
  $$
    \rho(\{
      g\in G_n \,:\,
      T_n(\pi_gX^{(n)}) > t_{\omega,\veps}
    \}) \rightarrow 0 
  $$
  Hence, for all $\omega\in\Omega$ such that
  ${T(X(\omega)) \le t}$,
   $\rho(\{
    g\in G_n \,:\,
    T_n(\pi_gX^{(n)}(\omega)) > t
  \})\rightarrow 0$.
  As the concentration inequality is agnostic to
  direction, the above argument can be
  redone for $
    1 - \rho(\{
      g\in G_n \,:\,
      T_n(\pi_gX^{(n)}) > t
    \}) = 
    \rho(\{
      g\in G_n \,:\,
      T_n(\pi_gX^{(n)}) \le t
    \})
  $
  to conclude that 
  $
    \rho(\{
      g\in G_n \,:\,
      T_n(\pi_gX^{(n)}) > t
    \}) \convas \indc{T(X) > t}.
  $
  
  For convergence in $L^1$, it is trivial to note that 
  $\sup_n \abs{ \rho(\{
      g\in G_n \,:\,
      T_n(\pi_gX^{(n)}) > t
    \})} \le 1$.  Thus, the sequence 
    $\rho(\{
      g\in G_n \,:\,
      T_n(\pi_gX^{(n)}) > t
    \})$ is uniformly integrable and 
    converges almost surely, and hence in probability,
    from the first part of this lemma.  Hence,
    $$
    \xv\abs*{
    \rho(\{
      g\in G_n \,:\,
      T_n(\pi_gX^{(n)}) > t
    \}) - \indc{T(X) > t}
    } \rightarrow 0 
    $$
    by Theorem~10.3.6 of \cite{DUDLEY2002}.
\end{proof}

\begin{corollary}
  \label{cor:asympSize}
  Under the setting of Theorem~\ref{thm:asympInvariance},
  $$
    \lim_{n\rightarrow\infty}
    \prob{
      T_n(X^{(n)}) > t_\alpha(X^{(n)})  
    } \le \alpha.
  $$
\end{corollary}
\begin{proof}
  From Theorem~\ref{thm:asympInvariance}, 
  for any $\veps>0$, there exists an $N\in\natural$
  such that for all $n>N$,
  \begin{align*}
    \prob{ T_n(X^{(n)}) > t_\alpha(X^{(n)}) } 
    &=
    \xv\left[
      \prob{
        T_n(X^{(n)}) > t_\alpha(X^{(n)})  
        \,\mid\, \mathcal{S}_n\cap\mathcal{T}_n
      }
    \right]\\
    &\le 
    \xv\left[
      \rho\left(\{
      g\in G_n \,:\,
      T_n(\pi_gX^{(n)}) > t_\alpha(X^{(n)})
    \}\right)
    \right] + \veps
    \le \alpha + \veps.
  \end{align*}
  Taking $\veps\rightarrow0$ finishes the proof.
\end{proof}

\subsection{Remark on Group Selection}
\label{sec:groupSelection}

The above theorems and corollaries can hold for 
a multitude of groups.  In particular, if they
hold for a group $G$, then they hold for any 
subgroup of $G$.  
The choice of $G$ directly results in a choice of 
$\mathcal{S}$, the $\sigma$-field of invariant 
sets.  Indeed, a ``larger'' group $G$ will make
$\mathcal{S}$ ``smaller'', and thus, 
the randomized $\Indc{ T_n(\pi_gX^{(n)}) > t }$
can be used
to extract more information about 
$\Indc{ T_n(X^{(n)}) > t }$.
When conditioning on $\mathcal{S}$, the smaller 
$\mathcal{S}$ is, the more restricted the conditional 
probability will be.

For illustrative purposes, let $G$ be the 
trivial group.  In such a scenario, 
the random measures 
$
  \rho\left(\{
    g\in G_n \,:\,
    T_n(\pi_gX^{(n)}) > t
  \}\right)
$
and 
$
  \prob{
    T_n(X^{(n)}) > t  \,|\,
    \mathcal{S}_n\cap\mathcal{T}_n
  }
$
coincide with 
$\Indc{ T_n(X^{(n)}) > t }$
and no meaningful inference is achievable.
In particular, the randomization threshold
for a fixed $x$
is $t_\alpha(x) = T(x)$, and the conclusion of
Corollary~\ref{cor:asympSize} is the immensely unhelpful 
fact that $\lim_{n\rightarrow \infty} 0 \le \alpha$.

For a richer discrete group $G$ with cardinality
$\abs{G}$, the random measure 
$\rho(\{
    g\in G\,:\, T(\pi_gX)>t
\})$
can take on at most the values 
$i/\abs{G}$ for $i=0,1,\ldots,\abs{G}$.
Hence, the finer granularity of, say, the 
symmetric group over the alternating group
or the cyclic group may be preferable.  

In a statistical context,  
group selection for randomization tests 
is intimately connected 
to the null hypothesis under examination.
The tail probability $\prob{
      T(X) > t_\alpha(X)
    }$
from the above corollaries 
corresponds to a p-value concerned with whether
or not condition C2 holds.  That is,
the p-value is for the following hypotheses:
$$
  H_0: T(X) \eqdist T(\pi_g X)~\forall g\in G,
  ~~~~
  H_1: \exists g\in G~ s.t.~
  T(X) \overset{\text{d}}{\ne} T(\pi_g X).
$$
One approach may be to select the maximal 
invariant group of transformations that preserve
the distribution of $T(X)$ under the null hypothesis.
However,
the recent work of \cite{KONING_HEMERIK_2023}
proposes a clever approach to subgroup selection
for improving the performance of the
classical permutation test, which opens up
more research questions into the best choice
of $G$.

\subsection{Remark on Lipschitz constants}
\label{sec:lipSelection}

As far as Theorem~\ref{thm:asympInvariance} is 
concerned, the sequence of Lipschitz constants
$c_n$ can be arbitrary as long as the conditions
continue to hold.  Thus, the choice of $c_n$
is dependent on the problem under consideration.

In Section~\ref{sec:OneTest}, a simple example of
a one-sample location test is considered
for \iid $X_1,\ldots,X_n$.  
The function $T$ is chosen to be
$T(X^{(n)}) = n^{-1/2}\sum_{i=1}^nX_i$ in order
to contrast the result with the Berry-Esseen
Theorem.  However, the standard sample mean
$\bar{X} = n^{-1}\sum_{i=1}^nX_i$ is another
valid choice for $T$, which would, in contrast
to the $n^{-1/2}$-normalization, require weaker
moment assumptions on that random variables $X_i$.

Venturing deeper into this Lipschitzian rabbit hole,
one could easily consider the function $T$ defined as
$T(X^{(n)}) = n^{-2}\sum_{i=1}^nX_i$.  As far as
Theorem~\ref{thm:asympInvariance} is concerned,
this function is perfectly valid.  Of course,
it is well known that $T(X^{(n)})\convas0$
as long as $\xv\abs{X_1}^{1/2}<\infty$
and, of note,
irregardless of the mean of $X_1$
\cite[Theorem 3.2.3]{STOUT1974}.
Thus, this choice of $T$ is perfectly useless 
for statistical inference purposes.
The conclusion is that $T$ must be chosen
preemptively for the problem at hand and 
not to post-hoc make Theorem~\ref{thm:asympInvariance}
applicable.
Such selection of $c_n$ occurs in the subsequent 
Section~\ref{sec:lpballs} where $c_n$ is chosen
to achieve the right convergence properties.

In summary, the requirement of Theorem~\ref{thm:asympInvariance}
that $c_n\norm{ X^{(n)} }/\sqrt{n}\convas0$
dictates the relation between $c_n$ and $X^{(n)}$.
The faster $c_n$ tends to zero, the fewer conditions
are required on the moments of $X^{(n)}$.
However, if $c_n$ decreases too fast, the utility 
of $T$ is lost.  And furthermore, if $X^{(n)}$ is
a well-behaved random vector, then there are many
more valid choices of $c_n$.

\section{Uniform points in an $\ell_p^n$-ball}
\label{sec:lpballs}

As a toy application of Theorem~\ref{thm:asympInvariance},
convergence properties can be derived for sums of 
coordinates for random vectors within an $\ell_p^n$-ball.
The goal of this section is to show a 
law of the iterated logarithm
for uniformly random points in an $\ell_p^n$-ball, 
i.e. uncorrelated and bounded but not independent.
Theorems~\ref{thm:condProb} and~\ref{thm:asympInvariance}
allow for a quick proof of an almost sure upper 
bound with correct asymptotic rate, but sub-optimal
constant dependent on the concentration behaviour
of random group elements.

Let $X^{(n)}$ be a uniformly random point inside the 
$\ell_p^n$-ball, i.e. $\sum_{i=1}^n \abs{X_i}^p\le1$,
and function $T_n(X^{(n)}) = c_n\iprod{\boldsymbol{1}_n}{X^{(n)}}$
for some normalizing constant depending on $n$ such as
$c_n = n^{-1/2}$ and the $n$-long vector
$\TT{\boldsymbol{1}}_n = (1,\ldots,1)$.  
Thus, $T_n:\real^n\rightarrow\real$
is $c_n$-Lipschitz.  Two simple examples follow.

\begin{example}[$\ell_\infty$ and $\ell_p$-balls for $p\le1$]
  Let $X^{(n)}$ be a uniformly random point within the 
  $\ell_\infty^n$-ball, i.e. 
  $\max_{i=1,\ldots,n}\abs{X_i}\le1$. 
  Then, it is well known via the central limit 
  theorem, Chebyshev's inequality, and the 
  strong law of large numbers, respectively,
  that 
  $$
    \sum_{i=1}^n \frac{X_i}{n^{1/2}}
    \convd Z\dist\distNormal{0}{1},
    ~~
    \sum_{i=1}^n \frac{X_i}{n^q}
    \convp 0,
    \text{ and }
    \sum_{i=1}^n \frac{X_i}{n^{q+1/2}}
    \convas 0
  $$
  for any choice of $q>1/2$.
  Furthermore, as the collection of $X_i$ 
  are \iid Uniform$[-1,1]$ random variables,
  the Law of the Iterated Logarithm 
  (see, for example, \cite{deACOSTA1983} or 
  \cite{LEDOUXTALAGRAND1991} Chapter 8) 
  implies that
  $$
    \limsup_{n\rightarrow\infty}
    \sum_{i=1}^n\frac{X_i}{\sqrt{2n\log\log n}}
    = 1 \text{ and }
    \liminf_{n\rightarrow\infty}
    \sum_{i=1}^n\frac{X_i}{\sqrt{2n\log\log n}}
    = -1
  $$
  and thus
  $
    n^{-q}\sum_{i=1}^n {X_i}
    \convas 0
  $
  for any choice of $q>1/2$.

  Let $X^{(n)}$ be a uniformly random point within the 
  $\ell_p^n$-ball for $p\le1$.  Then, 
  $\sum_{i=1}^n \abs{X_i} \le 1$.  Hence,
  $$
    \sum_{i=1}^n \frac{X_i}{n^q}
    \convp 0
  $$
  for any choice of $q>0$.
  Furthermore, a quick calculation\footnote{
    See Appendix~\ref{app:lpballs} for tedious details.
  } 
  for the 
  $\ell_1^n$-ball shows that 
  $\xv X_1^2 \le C/n^2$ for some constant $C>0$.
  Thus,
  $
    {n^{-q}}\sum_{i=1}^n {X_i}
    \convas 0
  $
  for any $q>0$ via Chebyshev's inequality 
  and the first Borel-Cantelli lemma.
  The goal of what follows is to extend this 
  idea to other $\ell_p^n$-balls
  and achieve more precise rates of convergence.
\end{example}

The group of interest is the 
special orthogonal group, 
$SO(n) = \{ M\in \real^{n\times n} \,:\, 
\TT{M}M = M\TT{M} = I,\,\det(M)=1 \}$.
Integrating over $SO(n)$ with respect to its normalized
Haar measure $\rho$ yields
$$
  \xv_{SO(n)} T_n(MX^{(n)}) = 0
  \text{ and }
  \xv_{SO(n)} T_n(MX^{(n)})^2 = c_n^2\norm{X}_2^2.
$$
Indeed, for the second moment calculation, 
$$
  \xv_{SO(n)} T_n(MX^{(n)})^2 =
  c_n^2\xv_{SO(n)}\left[
    \TT{ (MX^{(n)}) }\boldsymbol{1}_n\TT{\boldsymbol{1}}_n
    (MX^{(n)})
  \right],
$$
and the calculation becomes a consequence of the
following lemma.

\begin{proposition}
  \label{lem:rotInnerProduct}
  Let
  $A\in\real^{n\times n}$ be a symmetric
  matrix with spectrum $\lmb_1\ge\lmb_2\ge\ldots\ge\lmb_n$, 
  and let
  $b_A:\real^n\times\real^n\rightarrow \real^+$
  be a bilinear form defined by 
  $$
    b_{A,SO(n)}(x,y) = 
    \int_{SO(n)}
    \TT{(Mx)}A(My) d\rho(M)
  $$
  where integration is taken over $SO(n)$ with 
  respect to Haar measure $\rho$.
  Then, $b_{A,SO(n)}(x,y)$ is rotationally invariant
  and furthermore
  $$
    b_{A,SO(n)}(x,y) = \bar{\lmb}\iprod{x}{y}
  $$
  where $\iprod{\cdot}{\cdot}$ is the standard Euclidean 
  inner product and $\bar{\lmb} = n^{-1}\sum_{i=1}^n\lmb_i$. 
\end{proposition}

\begin{proof}
  Any bilinear form on 
  a real Hilbert Space is of the form $\iprod{Mx}{y}$ for
  some bounded operator $M$.
  Hence, $b(x,y) = \sum_{i=1}^ncx_iy_i$ for
  some $c>0$ by rotational invariance.  
  Without loss of generality, let $A$ be diagonal 
  with entries $\lmb_1,\ldots,\lmb_n$.  Then, choosing $x=y$ to
  be any unit vector results in 
  $$
    c = \int_{\norm{v}=1} \TT{v}Av dv
      = \int_{\norm{v}=1} \sum_{i=1}^n \lmb_i v_i^2 d\mu
  $$
  where $\mu$ is the uniform surface measure of the $(n-1)$-sphere. 
  By symmetry, the integral can be restricted to fraction 
  of the sphere where $v_i\ge0$.  Furthermore, $\{v_i=v_j\,:\,i\ne j\}$
  is a measure zero event.  Thus,
  \begin{align*}
    c &= 2^n
    \int_{v_1>\ldots>v_n\ge0} \sum_{\pi\in\mathbb{S}_n} 
      \sum_{i=1}^n \lmb_i v_{\pi(i)}^2 d\mu\\
    &= 2^n\int_{v_1>\ldots>v_n\ge0}  
       \sum_{i=1}^n \lmb_{i} \sum_{\pi\in\mathbb{S}_n} v_{\pi(i)}^2  d\mu\\
    &= 2^n\int_{v_1>\ldots>v_n\ge0}  
       \sum_{i=1}^n \lmb_{i}(n-1)!  d\mu\\
    &= \frac{2^n}{n!2^n}\sum_{i=1}^n \lmb_{i}(n-1)! = 
    \frac{1}{n}\sum_{i=1}^n \lmb_{i}
  \end{align*}
  as the sum is over $n!$ permutations in $\mathbb{S}_n$, 
  which is grouped into $(n-1)!$ sets of $n$ $v_i^2$'s 
  that sum to 1.
\end{proof}

Let $f_{n,x}:SO(n)\rightarrow\real$ be defined as 
$f_{n,x}(M):= T_n(Mx)$.  Then, as a consequence of
the Cauchy-Schwarz inequality, 
$$
  \abs{ f_{n,x}(M_1) - f_{n,x}(M_2) }
  \le c_n\sqrt{n}\norm{x}_2\norm{M_1-M_2}_\text{HS}
$$
making $f_{n,x}$ a 
$c_n\sqrt{n}\norm{x}_2$-Lipschitz function on 
$SO(n)$ with respect to the Hilbert-Schmidt metric.
Concentration of measure for the classical compact 
groups \citep[Theorem 5.17]{MECKES2019} implies that
$$
  \rho\left[
    f_{n,x}(M) \ge t
  \right] \le \exp\left[
    -\left(\frac{n-2}{n}\right)\frac{t^2}{24c_n^2\norm{x}_2^2}
  \right] 
  \le
  \exp\left[
    -\frac{\kappa_0 t^2}{c_n^2\norm{x}_2^2}
  \right]
$$
for  $n\ge3$ and some dimension independent constant
$\kappa_0>0$, e.g. $\kappa_0 = 1/72$.

In what follows, it is shown that 
\begin{equation}
  \label{eqn:lilConj}
  \limsup_{n\rightarrow\infty}
  \frac{
    \abs*{\sum_{i=1}^{n} X_i}
  }{Kn^{1/2-1/p}\sqrt{ \log\log n }}
  \le 1, ~~\text{a.s.}
\end{equation}
for some constant $K>0$.
Our conjecture is $K = 2^{1/2+1/p}$.
Figure~\ref{fig:lilBall} displays the
value of 
$
  {
    \abs*{\sum_{i=1}^{n} X_i}
  }/{2^{1/2+1/p}n^{1/2-1/p}\sqrt{ \log\log n }}
$
computed for 1000 simulated uniform vectors 
within the $\ell^n_p$-ball for $p = 1,2,\infty$
and $n = 10^1,10^2,\ldots,10^6$.
An algorithm for simulating such random vectors
is detailed in Appendix~\ref{app:ballSim}.

\begin{figure}
  \centering
  \includegraphics[width=0.75\textwidth]{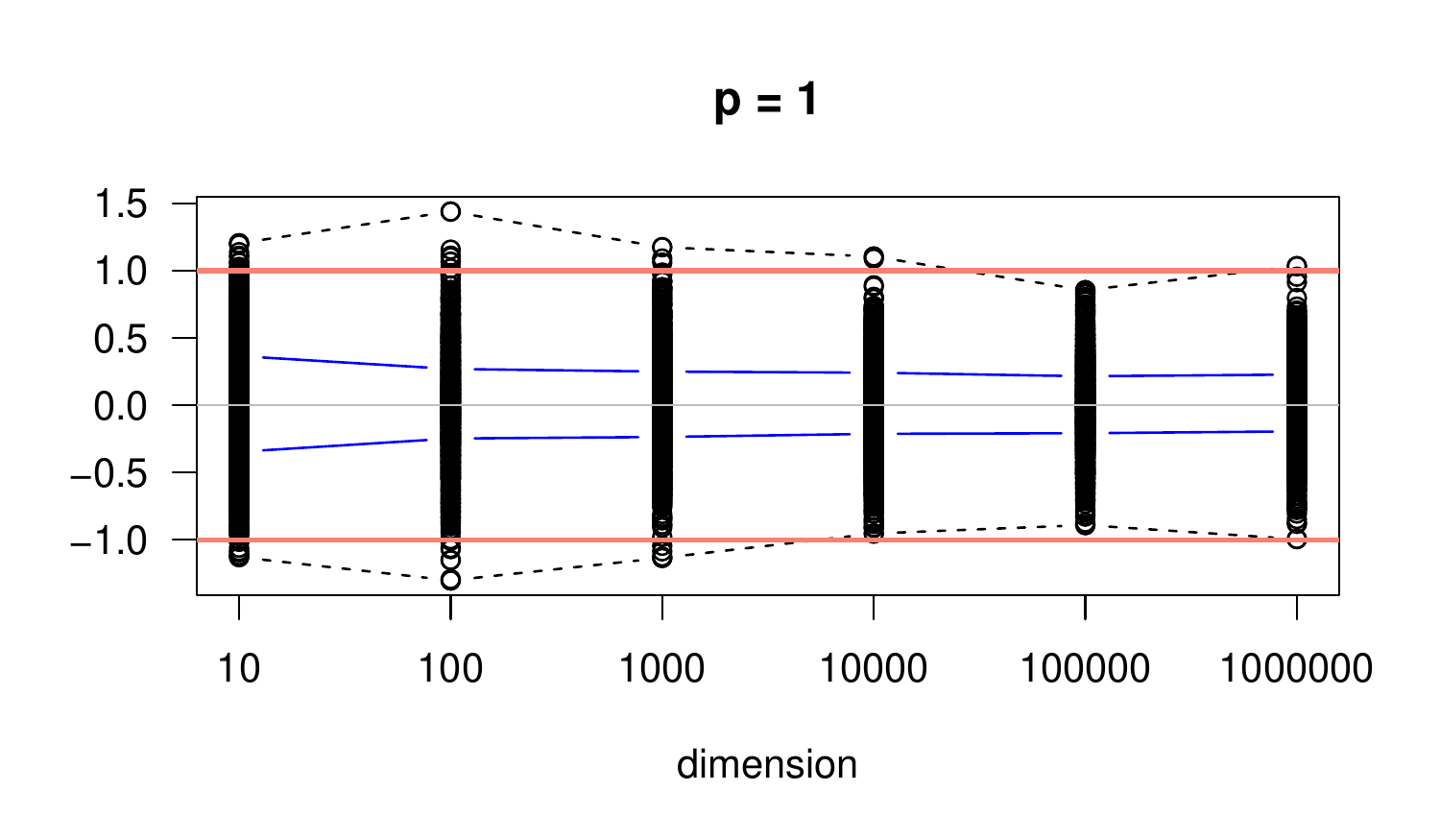}
  \vspace{-0.25in}
  
  \includegraphics[width=0.75\textwidth]{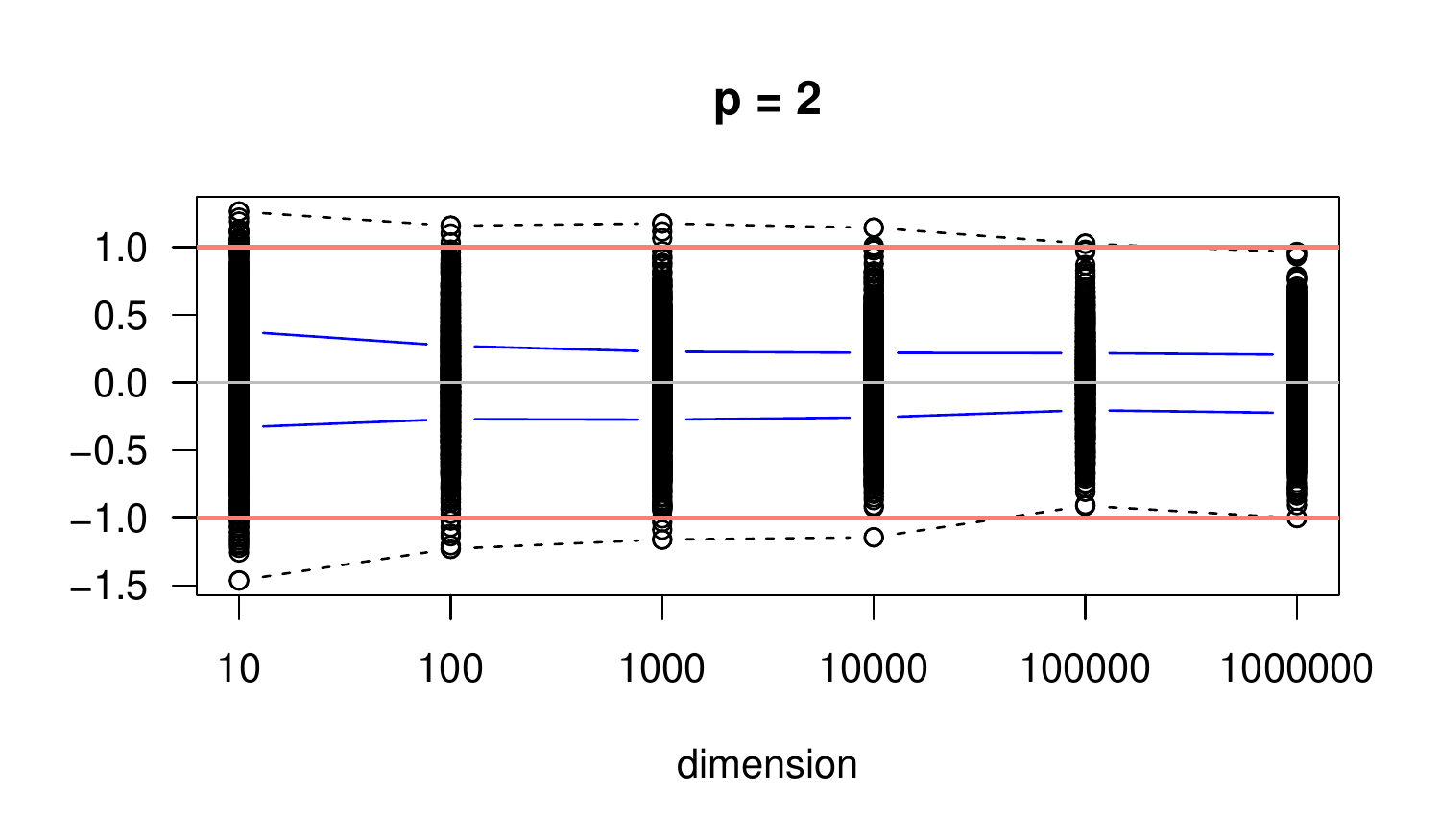}
  \vspace{-0.25in}
  
  \includegraphics[width=0.75\textwidth]{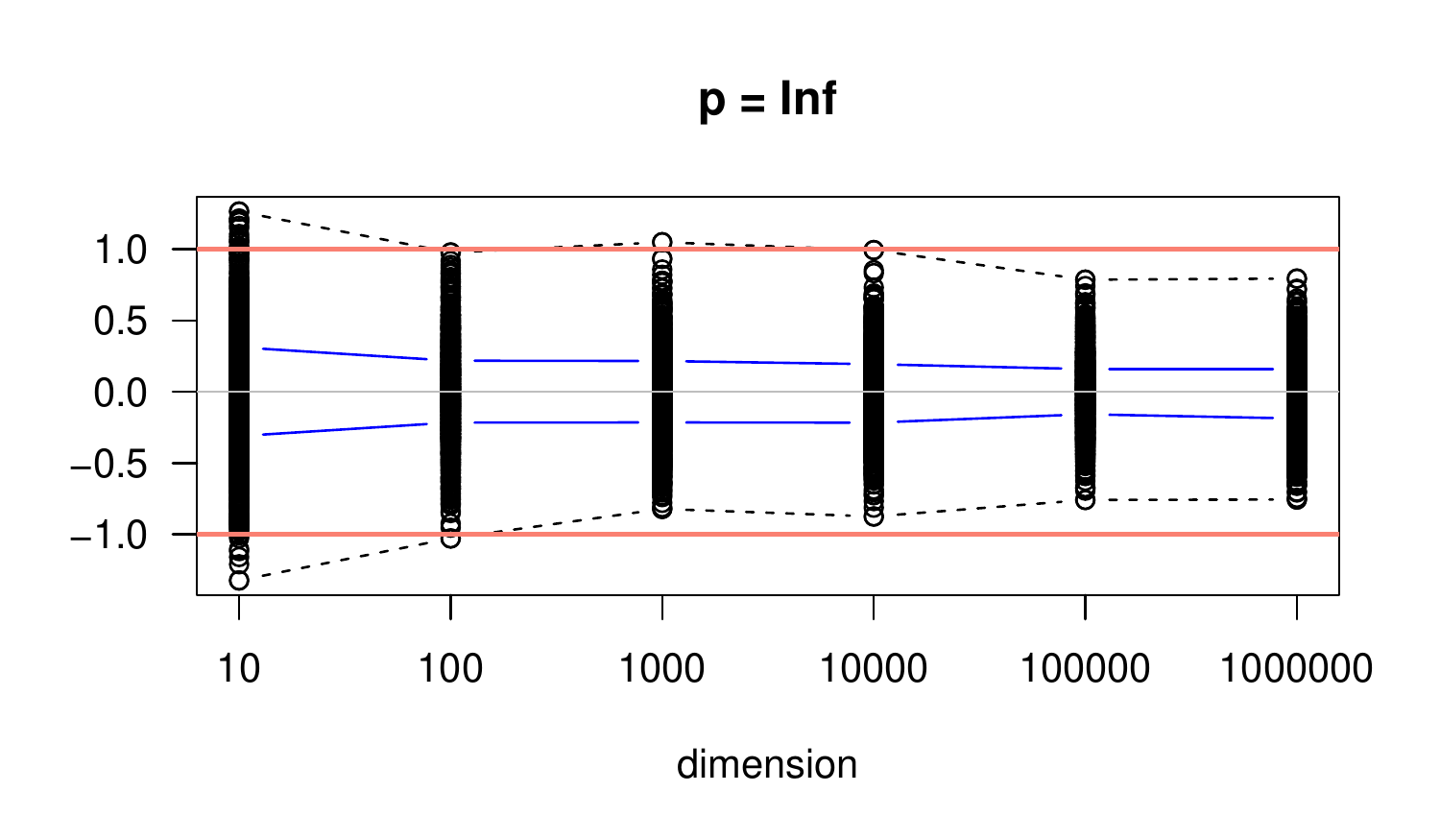}
  \vspace{-0.25in}
  
  \caption{
    \label{fig:lilBall}
    For dimensions $n=10^1,\ldots, 10^6$, 1000 vectors
    are generated from a uniform distribution on the 
    $\ell^n_p$-ball, and for each vector, 
    Equation~\ref{eqn:lilConj} with
    conjectured optimal constant is computed (black circles). 
    The red lines are for $\pm1$; the blue are for the 1st and 3rd
    quartile; the dashed black are for the min and max values.
  }
\end{figure}

\subsection{Case $p=2$}

If $X^{(n)}$ is a uniformly random point in the 
$\ell_2^n$-ball, then the measure induced by $X^{(n)}$
is invariant to any rotation $M\in SO(n)$.
Thus, Theorem~\ref{thm:condProb} and the above 
concentration of measure result imply that,
for any fixed $t>0$,
$$
  \prob{
    n^{-q}\sum_{i=1}^n X_i > t
  } \le
  \ee^{-\kappa_0 n^{2q}t^2}
$$
where $c_n = n^{-q}$.  Thus, $n^{-q}\sum_{i=1}^n X_i$
converges to zero in probabiliy for any $q>0$.
Furthermore, from a 
standard application of the first Borel-Cantelli
Lemma, $n^{-q}\sum_{i=1}^nX_i\rightarrow0$ almost
surely for any choice of $q>1/2$ similar to the 
direct computation approach performed in the appendix.

Going further, it can be quickly shown that for some 
$K>0$,
$$
  \limsup_{n\rightarrow\infty}
  \frac{\abs*{\sum_{i=1}^n X_i}}{K\sqrt{ \log\log n }}
  \le 1
$$
using similar arguments as is
\cite{STOUT1974},
\cite{deACOSTA1983}, and others.
Indeed, fix $\veps>0$, and note that the sequence
$(X_i)_{i=1}^n$ is a symmetric sequence, and hence
L\'evy's maximal inequalities are applicable
\citep[Proposition 2.3]{LEDOUXTALAGRAND1991}.
Let $n_k = \lfloor c_0^k\rfloor$ 
for some $c_0>1$.
Then, 
\begin{multline*}
  \prob{
    \frac{
      \max_{n\le n_k}\abs*{\sum_{i=1}^{n} X_i}
    }{K\sqrt{ \log\log n_k }}
    \ge 1+\veps
  }
  \le 
  2\prob{
    \frac{\abs*{\sum_{i=1}^{n_k} X_i}}{K\sqrt{ \log\log n_k }}
    \ge 1+\veps
  } 
  \\ \le
  2\exp\left(
    -\kappa_0(1+\veps)^2K^2 \log\log n_k
  \right)
  \le
  2(\log n_k)^{-\kappa_0(1+\veps)^2K^2}.
\end{multline*}
Hence, the series 
$$
  \sum_{k=1}^\infty
  \prob{
    \frac{
      \max_{n\le n_k}\abs*{\sum_{i=1}^{n} X_i}
    }{K\sqrt{ \log\log n_k }}
    \ge 1+\veps
  } <\infty
$$
given that $K$ is chosen such that 
$\kappa_0K^2\ge1$.  Thus, by invoking the first 
Borel-Cantelli lemma and taking $\veps\rightarrow0$
results in 
$$
  \limsup_{n\rightarrow\infty}
  \frac{\abs*{\sum_{i=1}^{n} X_i}}{K\sqrt{ \log\log n }}
  \le 1, ~~\text{a.s.}
$$

\subsection{Case $p>2$}

If $X^{(n)}$ is a uniformly random point in the 
$\ell^n_p$-ball for $2<p<\infty$, then the induced
measure is not rotationally invariant.
However, Theorem~\ref{thm:asympInvariance}
can still be applied.
Indeed,
let $c_n = n^{1/p-1/2}/\sqrt{ \log\log n }$.
As $X^{(n)}$ is restricted to a 
compact set, $\xv\norm{X^{(n)}}^{p'}<\infty$
for any choice of $p'\ge1$.
Hence, 
$
  \sum_{n=1}^\infty 
  n^{ p'(1/p-1/2) }/\sqrt{\log\log n} <\infty
$
for any choice of $p' > (p-2)/2p$.
Furthermore,
$\norm{X^{(n)}} \le n^{1/2-1/p}$, so
$n^{-1/2}c_n\norm{X^{(n)}} \le 1/\sqrt{ n\log\log n }$, 
which converges to zero almost surely.
Therefore,  
$$
  \abs*{
    \prob{
      \frac{n^{1/p-1/2}}{\sqrt{\log\log n}}
      \sum_{i=1}^n X_i > t
      \,|\, \mathcal{S}_n\cap\mathcal{T}_n
    }
    - \rho\left[
      f_{n,X}(M) \ge t
    \right]
    } \rightarrow 0
$$
almost surely and in $L^1$
by Theorem~\ref{thm:asympInvariance}.  
Sub-Gaussian concentration on $SO(n)$
as a consequence of Theorem~5.17 from 
\cite{MECKES2019},
results in 
\begin{align*}
  &\rho\left[
    f_{n,X}(M) \ge t
  \right]
  \le
  \exp\left(
    -\frac{\kappa_0 t^2}{\norm{X^{(n)}}^2}
    n^{1-2/p}\log\log n
  \right)
  \le
  \ee^{
    -{\kappa_0 t^2}
    \log\log n
  }
  \text{ and}\\
  &\rho\left[
    c_n\iprod{\boldsymbol{1}_n}{(I-M)X^{(n)}} \ge t
  \right]
  \le
  \ee^{
    -{\kappa_0 t^2}
    \log\log n
  }.
\end{align*}
Noting that 
$
  \iprod{\boldsymbol{1}_n}{X^{(n)}} =
  \iprod{\boldsymbol{1}_n}{MX^{(n)}} +
  \iprod{\boldsymbol{1}_n}{(I-M)X^{(n)}},
$
the same argument as in the previous 
sub-section where $p=2$ can be applied to 
show the existence of some $K>0$ 
such that 
$$
  \limsup_{n\rightarrow\infty}
  \frac{
    \abs*{\sum_{i=1}^{n} X_i}
  }{Kn^{1/2-1/p}\sqrt{ \log\log n }}
  \le 1, ~~\text{a.s.}
$$

\subsection{Case $p<2$}

The case of a uniform point inside an 
$\ell_p^n$-ball with $1\le p<2$ poses some
additional problems to surmount.
Of note, the condition in Theorem~\ref{thm:asympInvariance}
that the normalizing sequence 
$c_n\rightarrow 0$, and furthermore that 
there exists a $p'\ge1$ such that 
$\sum_{i=1}^\infty c_n^{p'}<\infty$, fails to 
hold when the desired $c_n = n^{1/p-1/2}/\sqrt{\log\log n}$.
This suggests an alternative to Lemma~\ref{lem:asympInvarT}
when $\xv \norm{X^{(n)}}$ tends to zero as $n\rightarrow\infty$.
That is, noting that 
$$
    \abs{T_{n}(\pi_g X^{(n)}) - T_{n}(X^{(n)})}
    \le c_{n} \norm{\pi_gX^{(n)} - X^{(n)}} 
    \le 2 c_{n} \norm{X^{(n)}},
$$
allows for the condition $c_n\rightarrow0$ to 
be replaced by $c_n\xv\norm{X^{(n)}}\rightarrow0$
to conclude that 
$$
  \xv\abs{T_{n}(\pi_g X^{(n)}) - T_{n}(X^{(n)})}\rightarrow0.
$$
It similarly follows that $c_n\norm{X^{(n)}}\convas0$
implies that 
$\abs{T_{n}(\pi_g X^{(n)}) - T_{n}(X^{(n)})}\convas0$.
This can be applied to uniformly random points within
the $\ell^n_p$-ball for $1\le p<2$.

In this setting, a theorem of Schechtman 
and Zinn proposes concentration of the Euclidean 
norm on such an $\ell_p^n$-ball
\citep{SCHECHTMAN_ZINN_1990,SCHECHTMAN_ZINN_2000}.
That is, for $Y^{(n)}=(Y_1,\ldots,Y_n)$
a uniformly random point on the surface of an 
$\ell_p^n$-ball with $1\le p <2$, then there 
exist constants $T,c>0$ such that for all
$t> Tn^{1/2-1/p}$, 
$
  \prob{ \norm{Y} > t } \le \ee^{-cnt^p}.
$
This implies that 
\begin{align*}
  \xv\norm{X^{(n)}} 
  &\le \xv\norm{Y^{(n)}}
  = \int_{0}^{Tn^{1/2-1/p}} 
    \prob{ \norm{Y} > t } dt +
    \int_{Tn^{1/2-1/p}}^1 
    \prob{ \norm{Y} > t } dt\\
  &\le 
  Tn^{1/2-1/p} +
  \int_{Tn^{1/2-1/p}}^1 
   \ee^{-cnt^p}dt
\end{align*}
and the last integral, with 
$\delta_n = Tn^{1/2-1/p}$, is bounded as follows:
$$
  \int_{\delta_n}^1 \ee^{-cnt^p}dt
  \le 
  \int_{\delta_n}^1 \frac{t^{p-1}}{\delta_n^{p-1}} 
  \ee^{-cnt^p}dt
  \le  
  \frac{n^{1/2-p/2-1/p}}{T^{p-1}cp}.
$$
The result is
$$
  \xv\norm{ X^{(n)} } \le
  n^{1/2-1/p}\left(
    T + \frac{1}{T^{p-1}cpn^{p/2}}
  \right) = 
  O( n^{1/2-1/p} ).
$$

Thus, in the context of the problem at hand,
$c_n\xv\norm{X^{(n)}} = O( (\log\log n)^{-1/2} )$
and via Markov's inequality,
$c_n\norm{X^{(n)}}\convp0$, and 
via the L\'evy-It\^o-Nisio theorem
\citep[Theorem 2.4]{LEDOUXTALAGRAND1991},
$c_n\norm{X^{(n)}}\convas0$.
Thus, $T_n(X^{(n)})$ is once again asymptotically
invariant to rotations, 
and invoking the same argument as before 
shows the existence of some $K>0$ 
such that 
$$
  \limsup_{n\rightarrow\infty}
  \frac{
    \abs*{\sum_{i=1}^{n} X_i}
  }{Kn^{1/2-1/p}\sqrt{ \log\log n }}
  \le 1, ~~\text{a.s.}
$$

\section{One and two sample testing}
\label{sec:specificTests}

A different application of Theorems~\ref{thm:condProb}
and~\ref{thm:asympInvariance} is statistical 
hypothesis testing.  In the following sections,
classical examples of one and two sample testing
are considered.  For comparison with each case, 
quantitative versions of Theorem~\ref{thm:asympInvariance}
are proved assuming third moment conditions, 
which rely on results like the Berry-Esseen theorem.
The constants in the following theorems can likely
be improved with more careful arguments, but
this is not investigated in this work.  These
results are included for contrast with
the previously mentioned theorems and not assumed
to be state-of-the-art.

\subsection{One Sample Location Test}
\label{sec:OneTest}

A simple example of asymptotic invariance arises in the one sample
location test (see \cite{LEHMANN2006} examples 15.2.1, 15.2.4, and 15.2.5).
Given $X = (X_1,\ldots,X_n)$ \iid real valued random variables with 
mean $\mu$, the hypotheses under consideration are
$$
  H_0:~\mu=0 ~~\text{ and }~~ H_1:~\mu\ne0.
$$
Let $G = \{-1,+1\}^{n}$ be the group corresponding to 
the vertices of the 
$n$-dimensional hypercube.
For $\{\theta_i\}_{i=1}^n$ such that $\sum_{i=1}^n\theta_i^2=1$,
Let $T:\real^n\rightarrow\real$ be $T(x)=\sum_{i=1}^n\theta_ix_i$.
And lastly, let $\pi_gx = (\pm x_1, \ldots, \pm x_n)$.
To apply a randomization test based on the the group $G$,
the additional assumption that the univariate distribution 
of the $X_i$ is symmetric about the origin is required, i.e. 
$\prob{X_i\in B} = \prob{X_i\in -B}$ where 
$-B = \{x\in\real\,:\, -x\in B\}$.  In which case, 
condition C2 from above is satisfied, i.e. 
$T(\pi_gX) = T(X)$ in distribution, and thus the 
conclusions of
Theorem~\ref{thm:condProb} and Corollary~\ref{cor:condProb}
are valid.

Even in the absence of symmetry, Theorem~\ref{thm:asympInvariance}
outlines conditions under which significance thresholds 
based on random sign flips asymptotically achieve the 
desired test size.  In the above setting, 
the groups $G$ when paired with the Hamming metric 
form a normal L\'evy family \citep[Theorem~2.11]{LEDOUX2001}.
The function $T$
is Lipschitz with constant 
$c_n = \max_{i=1,\ldots,n}\abs{\theta_i}$.  Thus, 
the $\theta_i$ must be chosen such that for some $p\ge1$,
$\sum_{n=1}^\infty\max_{i=1,\ldots,n}\abs{\theta_i}<\infty$.
Moreover, the condition that 
$c_n\norm{X^{(n)}}/\sqrt{n}\convas0$ has further implications
on the moments of $\norm{X^{(n)}}$ and the choice of 
$\theta_i$.  By invoking Theorem 2.1.3. of \cite{STOUT1974},
almost sure convergence to 0 is achieved if for
some $p>0$,
$$
  \sum_{i=1}^\infty \frac{c_n^p\xv\norm{X^{(n)}}^p}{n^{p/2}}<\infty.
$$
In particular, as the $X_i$'s are iid, 
choosing $p = 2q$ results in a simplified convergence condition:
$$
  \sum_{n=1}^\infty \frac{c_n^{2q}}{n^q}\xv\left(
    \sum_{i=1}^n X_i^2
  \right)^q  \le
  \xv[X_1^{2q}] \sum_{n=1}^\infty {c_n^{2q}}
  < \infty. 
$$
Hence, if $X_1$ has a finite $p$th moment, then the
$\theta_i$ must be chosen so that 
$\max_{i=1,\ldots,n}\abs{\theta_i} = o(n^{1/p})$.
For the common choice of 
$\theta_1 = \ldots = \theta_n = 1/\sqrt{n}$, 
Theorem~\ref{thm:asympInvariance} requires 
the relatively mild condition that 
$\xv X_1^{2+\veps}$ for some $\veps>0$ in order
to achieve asymptotic equivalence.

In contrast, assuming a finite third moment
allows for a quantitative version of 
Theorem~\ref{thm:asympInvariance} to be proved
directly for this specific setting.  
Indeed, a simple application of
the Berry-Esseen theorem \citep[ Section XVI.5]{FELLER2008B} under the appropriate assumptions
demonstrates the asymptotic validity of the randomization
test.  More recent work on Berry-Esseen bounds can be used
to generalize beyond the iid setting and make use of other
``natural characteristics'' beyond merely the third absolute
moment
\citep{BOBKOV2014,BOBKOV2018}, but this is not 
explored in this work.

\begin{theorem}
  \label{thm:oneSample}
  Let $X = (X_1,\ldots,X_n)$ be iid mean zero random 
  variables with variance $\sigma^2$ and 
  $\xv\abs{X_i}^3 = \omega<\infty$.  Then,
  for $T$ as above and 
  for all $t\in\real$ and some universal constant $C>0$,
  $$
    \abs*{
      \prob{T(X)>t} - 
      \xv_X\rho\left(
        \{
        g\in G\,:\, 
        T(\pi_gX)>t
        \}
      \right)
    } 
    \le
    \frac{2C\omega}{\sigma^3}\sum_{i=1}^n\abs{\theta_i}^3.
  $$
  Furthermore, if $\sum_{i=1}^n\abs{\theta_i}^3\rightarrow 0$ 
  as $n\rightarrow\infty$, then the probabilities coincide
  asymptotically.  In particular, if $\theta_1=\ldots=\theta_n=n^{-1/2}$, then
  the right hand side is $O(n^{-1/2})$.
\end{theorem}
\begin{proof}
  Let $\veps = (\veps_1,\ldots,\veps_n)\in\{-1,+1\}^n$.  
  Let $\Phi(t)$ be the cumulative distribution function 
  for a univariate standard normal random variable.
  In this setting, 
  \begin{align*}
    \xv_X\rho\left(
        \{
        g\in G\,:\, 
        T(\pi_gX)>t
        \}
      \right)
    &=
    2^{-n}\xv_x\left( \sum_{\veps\in\{-1,+1\}^n}
    \Indc{\sum_{i=1}^n\veps_i\theta_iX_i > t}\right)\\
    &=
    2^{-n}\sum_{\veps\in\{-1,+1\}^n}
    \prob{\sum_{i=1}^n\veps_i\theta_iX_i > t}
  \end{align*}
  Irregardless of $\veps_i$,
  $\xv \veps_i\theta_iX_i=0$, 
  $\xv (\veps_i\theta_iX_i)^2=\theta_i^2\sigma^2$, and
  $\xv \abs{\veps_i\theta_iX_i}^3=\abs{\theta_i}^3\omega$.  
  Consequently, the Berry-Esseen theorem
  \citep[Theorem 2, Section  XVI.5]{FELLER2008B}
  implies that there exists a universal constant $C>0$
  such that for any fixed choice of $\veps$
  $$
    \abs*{\prob{\frac{1}{\sigma}
    \sum_{i=1}^n\veps_i\theta_iX_i \le t} - \Phi(t)}
    \le
    \frac{C\omega}{\sigma^3}\sum_{i=1}^n\abs{\theta_i}^3.
  $$
  And thus, for $\Phi^c = 1- \Phi$,
  $$
    \abs*{
      \xv_X\rho\left(\{
        g\in G\,:\, 
        T(\pi_gX) > t
        \}
      \right)
      -
      \Phi^c(t/\sigma)
    } \le \frac{C\omega}{\sigma^3}\sum_{i=1}^n\abs{\theta_i}^3.
  $$
  Finally,
  \begin{multline*}
    \abs*{
      \prob{T(X)>t} - 
      \xv_X\rho\left(
        \{
        g\in G\,:\, 
        T(\pi_gX)>t
        \}
      \right)
    }\\
    \le
    \abs*{
     \prob{T(X) > t} - 
     \Phi^c(t/\sigma)
    } +
    \abs*{
      \xv_X\rho\left(
        \{
        g\in G\,:\, 
        T(\pi_gX) > t
        \}
      \right) -
      \Phi^c(t/\sigma)
    }
    \le 
    \frac{2C\omega}{\sigma^3}\sum_{i=1}^n\abs{\theta_i}^3.
  \end{multline*}
\end{proof}

\subsection{Two Sample t-Test}
\label{sec:TwoTest}

The two sample t-test stands as a prototypical hypothesis
test \cite[Section 11.3]{LEHMANN2006}.  
The goal is to determine if two populations have the
same mean.
Let $X = (X_1,\ldots,X_n,X_{n+1},\ldots,X_{n+m})$ be independent
Gaussian real valued random variables such that 
$$
  \xv X_i = \left\{
  \begin{array}{ll}
    \mu_1, & i\le n \\
    \mu_2, & i>n
  \end{array}
  \right.
  \text{ and }
  \var{X_i} = \left\{
  \begin{array}{ll}
    \sigma_1^2, & i\le n \\
    \sigma_2^2, & i>n
  \end{array}
  \right..
$$
The sample means are defined as 
$\bar{X}_1 = n^{-1}\sum_{i=1}^nX_i$ and 
$\bar{X}_2 = m^{-1}\sum_{i=n+1}^{n+m}X_i$.
The standard two sample t-test statistic under the 
assumption of homogeneous variances is to compute the test
statistic
$$
  T_\text{hom}(X) = 
  \frac{\bar{X}_1 - \bar{X}_2}{s_p\sqrt{n^{-1}+m^{-1}}}
  \text{ with }
  s_p^2 = \frac{(n-1)s_1^2 + (m-1)s_2^2}{n+m-2}
$$
where $s_p^2$ is the pooled estimator for the population
variance based on the sample variances $s_1^2$ and $s_2^2$
calculated for each population.
Under the null hypothesis that $\mu_1=\mu_2$, 
the test statistic, $T_\text{hom}(X)$, has a t-distribution 
with $n+m-2$ degrees of freedom.

If, however, the population variances are heterogeneous, then 
the above test statistic will not have a t-distribution under
the null.  This is the so-called \textit{Behrens–Fisher problem}.
A standard solution to this problem is to use 
Welch's t-test.  The test statistic in this case is
$$
  T_\text{het}(X) = \frac{
    \bar{X}_1 - \bar{X}_2
  }{
    \sqrt{
      {s_1^2}/{n} + {s_2^2}/{m}
    }
  }.
$$
The distribution under the null hypothesis of equal 
population means can be roughly approximated by a t-distribution
with degrees of freedom equal to 
$$
\frac{
    (s_1^2/n + s_2^2/m)^2
  }{
    {s_1^4}/[{n^2(n-1)}] +
    {s_2^4}/[{m^2(m-1)}]
  }.
$$
As, for example, when $n\rightarrow\infty$ with $m$ fixed,
the degrees of freedom tend towards $m-1$.

The standard two sample permutation test arises from 
the unnormalized difference of means
$T(X) = \bar{X}_1 - \bar{X}_2$ and
uniformly random permutations from $G = \mathbb{S}_{n+m}$,
the symmetric group on $n+m$ elements.  
Then, conditioned on the 
observed data $X=x$, one computes
$$
  \text{p-value} = \frac{\abs{\{\pi_g\,:\,T(\pi_g x) \ge T(x)\}}}{(n+m)!}.
$$
Of course, this is computationally infeasible.  Thus,
the typical solution is to sample some permutations 
$\{\pi_1,\ldots,\pi_r\}$
uniformly at random from $\mathbb{S}_{n+m}$ and compute
$$
  \text{p-value} \approx 
    \frac{1+\sum_{i=1}^r \Indc{
      T(\pi_i x) \ge T(x)
    }}{1+r}.
$$
The consequences of such sampling are discussed in 
\cite{HEMERIK2018}.  Otherwise, 
\cite{KASHLAK_KHINTCHINE2020} develops
analytic methods for 
computing exact permutation test p-values
for two-sample and $k$-sample tests for data
in Banach spaces
by making use of Khintchine-Kahane-type inequalities.

It can be seen that condition C2 holds in this setting
assuming that either the variances are homogeneous,
i.e. $\var{X_1} = \ldots = \var{X_{n+m}}$,
or the sample sizes are equal, i.e. $n=m$.
If both of these assumptions fail to hold,
Theorem~\ref{thm:asympInvariance} may still be 
applicable.
This conclusion is similar to the one-sample test 
setting of the previous section.
Indeed, the symmetric group with the Hamming metric
is a normal L\'evy family
\cite[Corrolary 4.3]{LEDOUX2001}. The function
$T$ is Lipschitz with constant 
$c_{n+m} = \min\{1/n,1/m\}$, which,
without loss of generality, taking 
$n\ge m := m_n$ as a function of $n$
gives $c_{n+m}=1/m_n$.  Furthermore, for
$p = 2q \ge 1$,
$$
  \left(\frac{c_{n+m}\norm{X^{(n+m)}}}{{\sqrt{n+m}}}\right)^{2q}
  =
  \frac{\left(\sum_{i=1}^{n+m}X_i^2\right)^q}{m_n^{2q}(n+m)^q}.
$$
Hence, Theorem 2.1.3 of \cite{STOUT1974} again implies 
almost sure convergence to zero of the above sequence if 
$$
  \xv[ X_1^{2q} ] \sum_{n=1}^\infty \frac{1}{m_n^{2q}} < \infty. 
$$
Thus, for proportional sample sizes $m_n = \lceil c_0n \rceil$
for $0<c_0<1$, convergence is achieved when 
$p = 2q = 1+\veps$ for some $\veps>0$ and 
$\xv[ X_1^{1+\veps} ]<\infty$.

For comparison with Theorem~\ref{thm:asympInvariance},
the following theorem quantitatively 
bounds how poorly a permutation 
test can perform when the assumption of exchangeability
is violated for Gaussian data.
The subsequent corollary passes through the Berry-Essen
bounds to achieve a quantitative version of 
Theorem~\ref{thm:asympInvariance} for non-exchangeable
random variables with finite third absolute moment.
Of note, if the sample sizes are 
proportional as discussed above, e.g. 
$m_n = \lceil c_0n \rceil$, then
the next Theorem concludes that 
$$
  \abs*{
    \prob{ T(X) > t } -
    \xv_X\rho( \{g\in G\,:\,T(\pi_gX) > t\} )
  } =
  O\left(
    \sqrt{\frac{\log n }{n}}
  \right).
$$
Though, it is worth future consideration as to 
whether or not the $\log n$ term is necessary.

\begin{theorem}
  \label{thm:twoSample}
  Let $X_1,\ldots,X_n$ be iid univariate random variables
  with mean $\eta$, variance $\sigma_1^2$, and finite 
  absolute third moment.  Similarly, let 
  $X_{n+1},\ldots,X_{n+m}$ be iid and independent of the 
  first collection with mean $\eta$, variance $\sigma_2^2$,
  and finite absolute third moment.
  Let $T(X) = n^{-1}\sum_{i=1}^nX_i - m^{-1}\sum_{i=n+1}^{n+m}X_i$.
  Then, for any $t\in \real$ 
  \begin{multline*}
    \abs*{
      \prob{ T(X) > t } -
      \xv_X\rho( \{g\in G\,:\,T(\pi_gX) > t\} )
    } \le \\
    \sqrt{
    2\left(
      \frac{1}{m} - \frac{1}{n}
    \right)(  
      \sigma_1^2 - \sigma_2^2
    )\log\left(
         \frac{nm/\sqrt{n^2-m^2}}{\sqrt{2\pi(  
      \sigma_1^2 - \sigma_2^2
    )}}+1
       \right)
  }
  +
  O\left( \sqrt{\frac{1}{n} + \frac{1}{m}} \right).
  \end{multline*}
\end{theorem}

\begin{lemma}
  \label{lem:levyProk}
  Let $\mu$ and $\nu$ be centred Gaussian measures on $\real$
  with variances $\sigma_1^2$ and $\sigma_2^2$, respectively, 
  with $\sigma_1^2 \ge \sigma_2^2 >0$.
  Then, denoting the L\'evy-Prokhorov metric by 
  $d_\mathrm{LP}$,
  $$
    d_\mathrm{LP}(\mu,\nu) \le 
     \sqrt{
       2(\sigma_1^2-\sigma_2^2)\log\left(
         \frac{1}{\sqrt{2\pi(\sigma_1^2-\sigma_2^2)}}+1
       \right)
     }.
  $$
\end{lemma}
\begin{proof}
   The L\'evy-Prokhorov metric is defined as 
   \begin{align*}
     d_\mathrm{LP}(\mu,\nu) 
     &= \inf\{
       \veps>0 \,:\, \mu(A) \le \nu(A_\veps) + \veps
       \text{ and }  \nu(A) \le \mu(A_\veps) + \veps
     \} \\
     &= \inf\{
       d_\mathrm{KF}(X,Y) \,:\, X\dist\mu, Y\dist\nu
     \}
   \end{align*}
   where 
   $d_\mathrm{KF}(X,Y) = \inf\{  
     \veps>0 \,:\, \prob{ \abs{X-Y}>\veps } < \veps
   \}$ is the Ky Fan metric for random variables $X$ and $Y$.

   Considering $(X,Y)$ bivariate Gaussian, the variance of 
   $X-Y$ is minimized when $\cov{X}{Y} = \sigma_2^2$ and
   then $\var{ X-Y } = \sigma_1^2-\sigma_2^2$.  
   Thus,
   \begin{align*}
     \prob{ \abs{X-Y}>\veps } 
     &= 2\int_\veps^\infty 
     \frac{1}{\sqrt{2\pi(\sigma_1^2-\sigma_2^2)}}
     \exp\left( \frac{-t^2}{2(\sigma_1^2-\sigma_2^2)} \right)dt\\
     &\le 
     \frac{2}{\veps}\sqrt{\frac{\sigma_1^2-\sigma_2^2}{2\pi}}
     \exp\left( \frac{-\veps^2}{2(\sigma_1^2-\sigma_2^2)} \right),
   \end{align*}
   and thus,
   $$
     d_\mathrm{LP}(\mu,\nu) \le \veps
     \text{ such that } 
     \exp\left( \frac{-\veps^2}{2(\sigma_1^2-\sigma_2^2)} \right)
     - \frac{\veps^2}{2}
     \sqrt{\frac{2\pi}{\sigma_1^2-\sigma_2^2}} = 0.
   $$
   The solution to this equation is 
   $$
     \veps = \sqrt{
       2(\sigma_1^2-\sigma_2^2)W\left(
         1/\sqrt{2\pi(\sigma_1^2-\sigma_2^2)}
       \right)
     }
     \le
     \sqrt{
       2(\sigma_1^2-\sigma_2^2)\log\left(
         \frac{1}{\sqrt{2\pi(\sigma_1^2-\sigma_2^2)}}+1
       \right)
     }
   $$
   where $W$ is Lambert's W function.
\end{proof}

\begin{lemma}
  \label{lem:twoSamp}
  Let $X = (X_1,\ldots,X_n,X_{n+1},\ldots,X_{n+m})$ be independent
  Gaussian real valued random variables such that 
  $$
  \xv X_i =
  \eta,~ \forall i
  \text{ and }
  \var{X_i} = \left\{
  \begin{array}{ll}
    \sigma_1^2, & i\le n \\
    \sigma_2^2, & i>n
  \end{array}
  \right.
  $$
  assuming without loss of generality that $n> m$.
  Let $T(X) = n^{-1}\sum_{i=1}^nX_i - m^{-1}\sum_{i=n+1}^{n+m}X_i$
  and $G=\mathbb{S}_{n+m}$ be the symmetric group on 
  $n+m$ elements.  Then,
  \begin{multline*}
    \abs*{
      \prob{ T(X) > t } -
      \xv_X\rho( \{g\in G\,:\,T(\pi_gX) > t\} )
    }\\ \le
    \sqrt{
    2\left(
      \frac{1}{m} - \frac{1}{n}
    \right)(  
      \sigma_1^2 - \sigma_2^2
    )\log\left(
         \frac{nm/\sqrt{n^2-m^2}}{\sqrt{2\pi(  
      \sigma_1^2 - \sigma_2^2
    )}}+1
       \right)
  }.
  \end{multline*}
\end{lemma}
\begin{proof}[Proof of Lemma \ref{lem:twoSamp}]
  The Gaussian measure induced by $T(X)$ has zero mean
  and variance $\sigma_1^2/n + \sigma_2^2/m$ and will 
  be denoted by $\mu$.  In turn, 
  the Gaussian measure averaged over the group $G$ is a 
  weighted mixture of $m$ Gaussian distributions
  and will be denoted $\nu = \sum_{j=0}^m w_j\nu_j$.
  All component measures have 
  zero mean and a variance
  of 
  $$
    \sigma_1^2\left(\frac{n-j}{n^2}+\frac{j}{m^2}\right) +
    \sigma_2^2\left(\frac{m-j}{m^2}+\frac{j}{n^2}\right)
    =
    \frac{\sigma_1^2}{n} + \frac{\sigma_2^2}{m} +
    j\left(
      {m^{-2}} - {n^{-2}}
    \right)(  
      \sigma_1^2 - \sigma_2^2
    )
  $$
  with hypergeometric weights 
  $
    w_j = {n \choose j}
    {m \choose m-j}
    {n+m \choose m}^{-1}.
  $
  From Lemma~\ref{lem:levyProk}, the L\'evy-Prokhorov
  metric between $\mu$ and $\nu_j$ for $j\ne0$ is
    \begin{align*}
    d_\mathrm{LP}(\mu,\nu_j)
    &\le 
     \sqrt{
       2j\left(
      {m^{-2}} - {n^{-2}}
    \right)(  
      \sigma_1^2 - \sigma_2^2
    )\log\left(
         \frac{1}{\sqrt{2\pi j(
      {m^{-2}} - {n^{-2}}
    )(  
      \sigma_1^2 - \sigma_2^2
    )}}+1
       \right)
     }\\
     &\le 
     \sqrt{
       2j\left(
      \frac{n^2-m^2}{n^2m^2}
    \right)(  
      \sigma_1^2 - \sigma_2^2
    )\log\left(
         \frac{nm/\sqrt{n^2-m^2}}{\sqrt{2\pi(  
      \sigma_1^2 - \sigma_2^2
    )}}+1
       \right)
     }
  \end{align*}
  Applying the triangle inequality and 
  Jensen's inequality finishes the proof:
\begin{align*}
    d_\mathrm{LP}(\mu,\nu) &\le
  \sum_{j=0}^m
    \frac{
      {n \choose j}{m \choose m-j}
    }{ {n+m \choose m} }
    \sqrt{ 
       2j\left(
      \frac{n^2-m^2}{n^2m^2}
    \right)(  
      \sigma_1^2 - \sigma_2^2
    )\log\left(
         \frac{nm/\sqrt{n^2-m^2}}{\sqrt{2\pi(  
      \sigma_1^2 - \sigma_2^2
    )}}+1
       \right)
     }\\
  &\le
  \sqrt{
    \left(\frac{nm}{n+m}\right)
    2\left(
      \frac{n^2-m^2}{n^2m^2}
    \right)(  
      \sigma_1^2 - \sigma_2^2
    )\log\left(
         \frac{nm/\sqrt{n^2-m^2}}{\sqrt{2\pi(  
      \sigma_1^2 - \sigma_2^2
    )}}+1
       \right)
  }\\
  &\le
  \sqrt{
    2\left(
      \frac{n-m}{nm}
    \right)(  
      \sigma_1^2 - \sigma_2^2
    )\log\left(
         \frac{nm/\sqrt{n^2-m^2}}{\sqrt{2\pi(  
      \sigma_1^2 - \sigma_2^2
    )}}+1
       \right)
  }.
  \end{align*}
\end{proof}

\begin{proof}[Proof of Theorem \ref{thm:twoSample}]
 The function $T$ can be written equivalently as
 $$
   T(X) = \frac{1}{n}\sum_{i=1}^n(X_i-\eta) -
   \frac{1}{m}\sum_{i=n+1}^{n+m}(X_i-\eta).
 $$
 As a result, for $i\le n$ and $i>n$, respectively, 
 \begin{align*}
   \xv\left[
     \left(\frac{X_i-\eta}{n}\right)^2
   \right] &= \frac{\sigma^2_1}{n^2} &
   \xv\left[
     \left(\frac{X_i-\eta}{m}\right)^2
   \right] &= \frac{\sigma^2_1}{m^2} \\
   \xv\left[
     \abs*{\frac{X_i-\eta}{n}}^3
   \right] &= \frac{\upsilon_1}{n^3} &
   \xv\left[
     \abs*{\frac{X_i-\eta}{m}}^3
   \right] &= \frac{\upsilon_2}{m^3}
 \end{align*}
 Thus, the Berry-Esseen Theorem \citep[Theorem 2, Section XVI.5]{FELLER2008B}
 states that for some universal constant $C>0$, 
 \begin{multline*}
   \abs*{
     \prob{
       \left[\frac{\sigma_1^2}{m}+\frac{\sigma_2^2}{n}\right]^{-1/2}T(X)
       \ge t
     } - \Phi(t)
   }\\ \le 
   C \frac{ 
     \frac{\upsilon_1}{n^2} + \frac{\upsilon_2}{m^2}
   }{
     \left(
     \frac{\sigma_1^2}{n} + \frac{\sigma_2^2}{m}
     \right)^{3/2}
   }
   = 
   C \frac{ 
     m^2{\upsilon_1} + n^2{\upsilon_2}
   }{
     \left(
       m{\sigma_1^2} + n{\sigma_2^2}
     \right)^{2}
   }\sqrt{
     \frac{\sigma_1^2}{n} + \frac{\sigma_2^2}{m}
    } = O\left(
      \sqrt{
        \frac{1}{n} + \frac{1}{m}
      }
    \right).
 \end{multline*}
 Thus, 
 let $Z \in\real^{n+m}$ be multivariate Gaussian with mean
 zero and covariance matrix with zero off-diagonal entries
 and main diagonal of $\sigma_1^2$ for the first $n$ entries
 and $\sigma_2^2$ for the final $m$ entries.
 \begin{align*}
   &\abs*{
      \prob{ T(X) > t } -
      \xv_X\rho( \{g\in G\,:\,T(\pi_gX) > t\} )
    }\\ 
    &~~~~~~\le
    \abs*{
      \prob{ T(X) > t } -
      \prob{ T(Z) > t }
    }\\ &~~~~~~+
    \abs*{
      \prob{ T(Z) > t } - 
      \xv_Z\rho( \{g\in G\,:\,T(\pi_gZ) > t\} )
    }
    \\ &~~~~~~+
    \abs*{
      \xv_Z\rho( \{g\in G\,:\,T(\pi_gZ) > t\} ) -
      \xv_X\rho( \{g\in G\,:\,T(\pi_gX) > t\} )
    }
 \end{align*}
 The first line is $O(n^{-1/2}+m^{-1/2})$ by the 
 Berry-Esseen theorem as discussed above.
 The second line is bounded by Lemma~\ref{lem:twoSamp}.
 The third line is also $O(n^{-1/2}+m^{-1/2})$ by the 
 Berry-Esseen theorem.  Indeed, the measure 
 $\xv_Z\rho( \{g\in G\,:\,T(\pi_gZ) > t\} )$ is 
 a weighted mixture of Gaussians as discussed in 
 the proof of Lemma~\ref{lem:twoSamp}.
\end{proof}

For comparison with Theorem~\ref{thm:twoSample}, 
the following result provides a quantitative bound
on the total variation distance between a Gaussian
average and the randomly permuted average.  Of note,
the inclusion of the term $\sqrt{(n-m)/(n+m)}$ 
necessitates a stronger asymptotic agreement between
$n$ and $m$; i.e. $m = cn$ is no longer sufficient.
Such a result suggests that a total variation 
version of Theorem~\ref{thm:asympInvariance} 
may be of future interest. 

\begin{theorem}
  \label{thm:tvTwoGauss}
  Let $X = (X_1,\ldots,X_n,X_{n+1},\ldots,X_{n+m})$ be independent
  Gaussian real valued random variables such that 
  $$
  \xv X_i =
  \eta,~ \forall i
  \text{ and }
  \var{X_i} = \left\{
  \begin{array}{ll}
    \sigma_1^2, & i\le n \\
    \sigma_2^2, & i>n
  \end{array}
  \right.
  $$
  assuming without loss of generality that $n\ge m$.
  Let $T(X) = n^{-1}\sum_{i=1}^nX_i - m^{-1}\sum_{i=n+1}^{n+m}X_i$
  and $G=\mathbb{S}_{n+m}$ be the symmetric group on 
  $n+m$ elements.  For measures $\mu$ and $\nu$ on 
  $(\real,\mathcal{B})$ defined by
  $\mu(B) = \prob{ T(X)\in B }$ and 
  $\nu(B) = \xv_X\rho( \{g\in G\,:\,T(\pi_gX)\in B\} )$ for 
  any $B\in\mathcal{B}$, 
  $$
    \norm{\mu-\nu}_\mathrm{TV}
    \le
  \frac{1}{2}
  \left(
    \frac{n-m}{n+m}
  \right)^{1/2}
  \abs{\sigma_2^2 - \sigma_1^2}^{1/2}
  \max\left\{
  \sqrt{
    {\frac{  
       1
      }{
      \sigma_2^2 + \sigma_1^2
      }
    }
  },
  \sqrt{
    {
      \frac{  
      1
      }{
      2\sigma_2^2
      }
    }
  }
  \right\}.
  $$
\end{theorem}
\begin{proof}[Proof of Theorem \ref{thm:tvTwoGauss}]
  For two equivalent measures, $\mu$ and $\nu$, the 
  Kullback–Leibler Divergence is 
  $
    D_\mathrm{KL}(\mu,\nu) = 
    \int \log(\frac{d\mu}{d\nu})d\mu,
  $
  and the symmetric KL-divergence is defined to be
  $H(\mu,\nu) = 0.5[D_\mathrm{KL}(\mu,\nu) + D_\mathrm{KL}(\nu,\mu)]$.
  From Pinkser's Inequality,
  $$
  \norm{\mu-\nu}_\mathrm{TV} := 
  \sup_{B\in\mathcal{B}}\abs{ \mu(B) - \nu(B) }
  \le
  \sqrt{
    \frac{1}{2} \min\{D_\mathrm{KL}(\mu,\nu),D_\mathrm{KL}(\nu,\mu)\}
  }
  \le 
  \sqrt{
    \frac{1}{2}H(\mu,\nu)
  }.
  $$
  
  For two independent centred 
  univariate Gaussian measures on $\real$, $\gamma_1$ and $\gamma_2$,
  with variances $\sigma_1^2$ and $\sigma_2^2$, the symmetric
  KL-divergence is 
  $$
  H(\gamma_1, \gamma_2) = 
  \frac{1}{2}D_{KL}(\gamma_1|\gamma_2) +
  \frac{1}{2}D_{KL}(\gamma_2|\gamma_1)
  = 
  \frac{\sigma_1^2}{4\sigma_2^2} + 
  \frac{\sigma_2^2}{4\sigma_1^2} - \frac{1}{2} 
  = \frac{1}{4}\left(
    \frac{\sigma_1}{\sigma_2} -
    \frac{\sigma_2}{\sigma_1}
  \right)^2
  $$
  
  In the context of the two-sample t-test for Gaussian 
  data, the measure $\mu$ induced by $T(X)$ has zero mean
  and variance $\sigma_1^2/n + \sigma_2^2/m$.  In turn, 
  the measure $\nu$ averaged over the group $G$ is a 
  weighted mixture of $m$ Gaussian distributions 
  denoted as $\nu = \sum_{j=1}^m w_j\nu_j$.
  All component measures $\nu_j$ have 
  zero mean and a variance
  of 
  $$
    \sigma_1^2\left(\frac{n-j}{n^2}+\frac{j}{m^2}\right) +
    \sigma_2^2\left(\frac{m-j}{m^2}+\frac{j}{n^2}\right)
    =
    \frac{\sigma_1^2}{n} + \frac{\sigma_2^2}{m} +
    j\left(
      {m^{-2}} - {n^{-2}}
    \right)(  
      \sigma_1^2 - \sigma_2^2
    )
  $$
  with hypergeometric weights 
  $
    w_j = {n \choose j}
    {m \choose m-j}
    {n+m \choose m}^{-1}.
  $
  The symmetric KL-divergence between $\mu$ and the 
  $j$th mixture component $\nu_j$ is 
    \begin{align*}
    H(\mu,\nu_j)
  &=\frac{
  j^2\left(
    {m^{-2}} - {n^{-2}}
  \right)^2(  
    \sigma_1^2 - \sigma_2^2
  )^2
  }{
  4\left[\frac{\sigma_1^2}{n} + \frac{\sigma_2^2}{m}\right]
  \left[
    \frac{\sigma_1^2}{n} + \frac{\sigma_2^2}{m} +
    j\left(
      {m^{-2}} - {n^{-2}}
    \right)(  
      \sigma_1^2 - \sigma_2^2
    )
  \right]
  } \\
  &=
  \frac{1}{4}\left(\frac{
    j^2\left(
    {m^{-2}} - {n^{-2}}
    \right)^2(  
    \sigma_1^2 - \sigma_2^2
    )^2
    [{\sigma_1^2}{n^{-1}} + {\sigma_2^2}{m^{-1}}]^{-2}
  }{
    1 + 
    j\left(
    {m^{-2}} - {n^{-2}}
    \right)(  
    \sigma_1^2 - \sigma_2^2
    )
    [{\sigma_1^2}{n^{-1}} + {\sigma_2^2}{m^{-1}}]^{-1}
  }\right)\\
  &=
  \frac{1}{4}\left(\frac{
    C_{n,m}^2j^2\left(
    {m^{-2}} - {n^{-2}}
    \right)^2
  }{
    1 + 
    C_{n,m}j\left(
    {m^{-2}} - {n^{-2}}
    \right)
  }\right).
  \end{align*}
  where 
  $C_{n,m} = (  
    \sigma_1^2 - \sigma_2^2
    )
    [{\sigma_1^2}{n^{-1}} + {\sigma_2^2}{m^{-1}}]^{-1}$
  for notational convenience.  
  
  Joint convexity of the KL-divergence implies that 
  $
    D_{KL}(\mu,\nu) \le
    \sum_{j=1}^M w_j D_{KL}(\mu,\nu_j),
  $
  which translates into the same for $H(\mu,\nu)$.  
Thus, considering the extreme cases of 
$n=m$ and $n\rightarrow\infty$ for $m$ fixed 
results in 
\begin{align*}
    H(\mu,\nu) &\le
  \sum_{j=0}^m
    \frac{
      {n \choose j}{m \choose m-j}
    }{ {n+m \choose m} }
    \frac{1}{4}\left(\frac{
    C_{n,m}^2j^2\left(
    {m^{-2}} - {n^{-2}}
    \right)^2
  }{
    1 + 
    C_{n,m}j\left(
    {m^{-2}} - {n^{-2}}
    \right)
  }\right)\\
  &\le
  \frac{1}{4}
  \sum_{j=0}^m
    \frac{
      {n \choose j}{m \choose m-j}
    }{ {n+m \choose m} }
    C_{n,m}j\left(
    {m^{-2}} - {n^{-2}}
    \right)\\
  &= \frac{1}{4}
    \frac{mn}{n+m}
    \left(
    {m^{-2}} - {n^{-2}}
    \right)
    \abs{  
      \sigma_2^2 - \sigma_1^2
    }
    \left[
      \frac{\sigma_1^2}{n} + \frac{\sigma_2^2}{m}
    \right]^{-1}\\
  &= \frac{1}{4}
    \left(
    \frac{n}{m} - \frac{m}{n}
    \right)\abs{  
      \sigma_2^2 - \sigma_1^2
    }
    \left[
      \left(1+\frac{m}{n}\right)\sigma_1^2 + 
      \left(1+\frac{n}{m}\right)\sigma_2^2
    \right]^{-1}\\
  &= \frac{1}{4}
    \left(
    \frac{n}{m} - \frac{m}{n}
    \right)\abs*{\frac{  
      \sigma_2^2 - \sigma_1^2
      }{
      \sigma_2^2 + \sigma_1^2
      }
    }
    \left[
      1 + 
      \frac{
        {m\sigma_1^2}/{n} + 
        {n\sigma_2^2}/{m}
      }{
        \sigma_2^2 + \sigma_1^2
      }
    \right]^{-1}\\
    &= \frac{1}{4}
    \left(
    n^2-m^2
    \right)\abs*{\frac{  
      \sigma_2^2 - \sigma_1^2
      }{
      \sigma_2^2 + \sigma_1^2
      }
    }
    \left[
      nm + 
      \frac{
        {m^2\sigma_1^2} + 
        {n^2\sigma_2^2}
      }{
        \sigma_2^2 + \sigma_1^2
      }
    \right]^{-1}\\
    &= \frac{1}{4}
    \left(
    \frac{n^2-m^2}{(n+m)^2}
    \right)\abs*{\frac{  
      \sigma_2^2 - \sigma_1^2
      }{
      \sigma_2^2 + \sigma_1^2
      }
    }
    \left[
      \frac{
        {(nm + m^2)\sigma_1^2} + 
        {(nm + n^2)\sigma_2^2}
      }{
        (\sigma_2^2 + \sigma_1^2)
        (n+m)^2
      }
    \right]^{-1}\\
    &= \frac{1}{4}
    \left(
    \frac{n-m}{n+m}
    \right)\abs*{\frac{  
      \sigma_2^2 - \sigma_1^2
      }{
      \sigma_2^2 + \sigma_1^2
      }
    }
    \left[
      \frac{nm + m^2}{(n+m)^2}
      \frac{\sigma_1^2}{\sigma_1^2+\sigma_2^2}
      +
      \frac{nm + n^2}{(n+m)^2}
      \frac{\sigma_2^2}{\sigma_1^2+\sigma_2^2}
    \right]^{-1}\\
    &\le \frac{1}{4}
    \left(
    \frac{n-m}{n+m}
    \right)\abs*{\frac{  
      \sigma_2^2 - \sigma_1^2
      }{
      \sigma_2^2 + \sigma_1^2
      }
    }
    \max\left\{
      2, 1 + \frac{\sigma_1^2}{\sigma_2^2}
    \right\}\\
    &= 
    \max\left\{
    \frac{1}{2}
    \left(
    \frac{n-m}{n+m}
    \right)\abs*{\frac{  
      \sigma_2^2 - \sigma_1^2
      }{
      \sigma_2^2 + \sigma_1^2
      }
    }
    ,
    \frac{1}{4}
    \left(
    \frac{n-m}{n+m}
    \right)\abs*{
      1 -
      \frac{  
      \sigma_1^2
      }{
      \sigma_2^2
      }
    }
    \right\}
  \end{align*}
  Applying Pinsker's inequality from above concludes the proof.
\end{proof}

\subsection{Simulation Experiments}
\label{sec:sims}

The following subsections contain brief simulation
experiments to illustrate the behaviour of the 
above hypothesis tests.

\subsubsection{One Sample Location Test}

To examine Theorem~\ref{thm:oneSample} in a simulation
setting, the exponentially modified Gaussian distribution
(EMGD)
will be considered.  
A random variable $Z$ is said to be EMGD if it can be
written as $Z = X+Y$ where $X$ is Gaussian, $Y$ is 
exponential, and $X$ and $Y$ are independent.
This convolution of Gaussian and exponential distributions 
has some popularity in modelling problems within
chemistry and cellular biology
\citep{GRUSHKA1972,GOLUBEV2010}.
In what follows, $X\dist\distNormal{0}{1}$ and 
$Y\dist\distExp{\lmb}$, and $Z$ will be centred
by $1/\lmb$ to have mean zero.

For sample sizes $n\in\{10,100\}$, samples of $Z_1,\ldots,Z_n$
were generated 200 times for each exponential rate parameter 
$\lmb \in \{\infty, 10, 1, 0.1, 0.01, 0.001\}$ where 
$\lmb = \infty$ corresponds to $Y=0$ almost surely.  
Thus, as $\lmb$ tends towards zero, the skewness of the 
centred EMGD increases.
For each set of simulated $Z_1,\ldots,Z_n$, 
a standard one-sample t-test was performed via the 
\texttt{t.test()} function in the \texttt{stats} R package
\citep{RCODE}.
Secondly, a randomization test was performed by generating
2000 random sign vectors $\veps\in\{\pm1\}^n$
to approximate the value of 
$\xv_X\rho\left(
  g\in G\,:\, 
  T(\pi_gX)>t
  \right)
$
from Theorem~\ref{thm:oneSample}.

Figure~\ref{fig:oneSamp10} and Figure~\ref{fig:oneSamp100}
display the 200 computed p-values for $n=10$ and $n=100$,
respectively.  The t-test p-value is plotted against the
randomization test p-value.  When $n=10$ and the
exponential rate parameter is small, i.e. the skewness
is large, the p-values produced by the two tests 
begin to disagree.  However, when $n=100$, the two 
tests produce nearly identical p-values
regardless of skewness.

\begin{figure}
    \centering
    \includegraphics[width=0.8\textwidth]{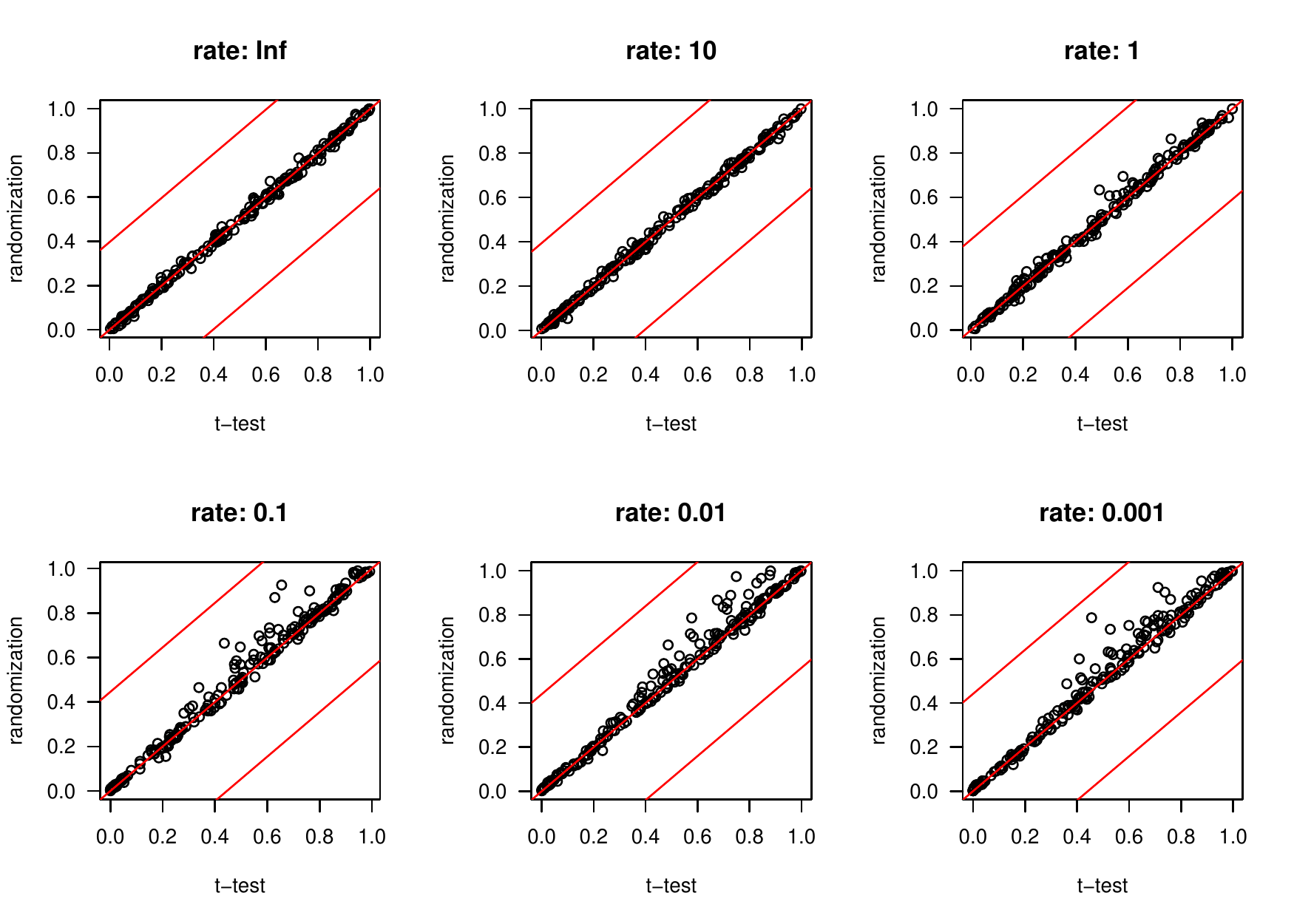}
    \caption{
      \label{fig:oneSamp10}
      A comparison of the two measures in Theorem~\ref{thm:oneSample}
      for $n=10$.  The red lines correspond to the Berry-Esseen 
      bounds.
    }
\end{figure}
    
\begin{figure}
    \centering
    \includegraphics[width=0.8\textwidth]{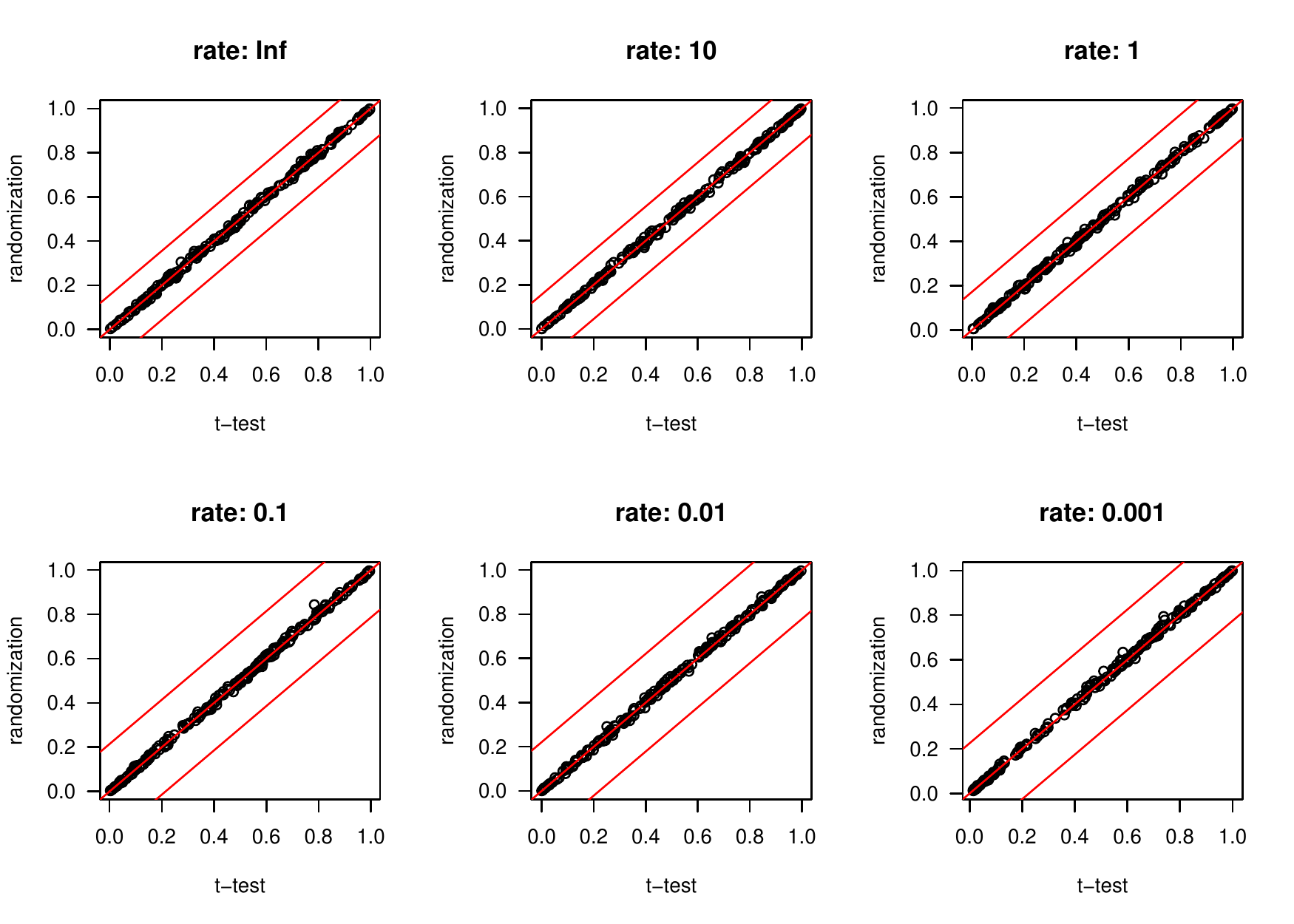}
    \caption{
      \label{fig:oneSamp100}
      A comparison of the two measures in Theorem~\ref{thm:oneSample}
      for $n=100$.  The red lines correspond to the Berry-Esseen 
      bounds.
    }
\end{figure}

\subsubsection{Two Sample t-Test}

To test the performance of the bound derived in 
Theorem~\ref{thm:twoSample}, the permutation test
and Welch's two sample t-test are compared on 
simulated data.
Figure~\ref{fig:cdfComp} compares the cumulative 
distribution functions of $T(X)$ and $T(\pi_gX)$
with respect to the bounds via the L\'evy-Prokhorov
metric for $(n,m)$ values of $(200,100)$,
$(2000,1000)$, and $(20000,10000)$ always in a 
two-to-one ratio.  
For each of these three choices for $(n,m)$, 
200 data vectors are simulated from a Gaussian 
distribution where the larger sample has a variance
of $1$ and the smaller sample has a variance of $16$.
As the total sample size
$N = n+m$ grows large, the difference between the
two cdf functions becomes vanishingly small.

Secondly,
for each of 200 replications, an iid standard Gaussian 
dataset is generated with sample size $n=200$, and 
iid Gaussian datasets
with mean zero, variance 16, and sample sizes $m=25,50,100,200$ are 
generated.  The p-value for 
Welch's two sample t-test, as described above,
is computed in R via the \texttt{t.test()}
function.  The permutation test p-value is computed
via 2000 permutations.
Figure~\ref{fig:twoSamp} displays the result of these
simulations.  When the smaller sample has a larger variance,
the permutation test is anti-conservative.  
That is, it produces p-values that are smaller than desired
and thus would lead to more frequent false rejections of the
null hypothesis.
But as $m$ approaches
$n$, both the discrepancy between the two tests and the 
bounds from Theorem~\ref{thm:tvTwoGauss} vanish.
If the sample with $n=200$ observations came from the population
with the larger 
variance, then the permutation test would instead be too 
conservative and the plots would be reflected across the
diagonal.

\begin{figure}
  \centering
  \includegraphics[width=0.31\textwidth]{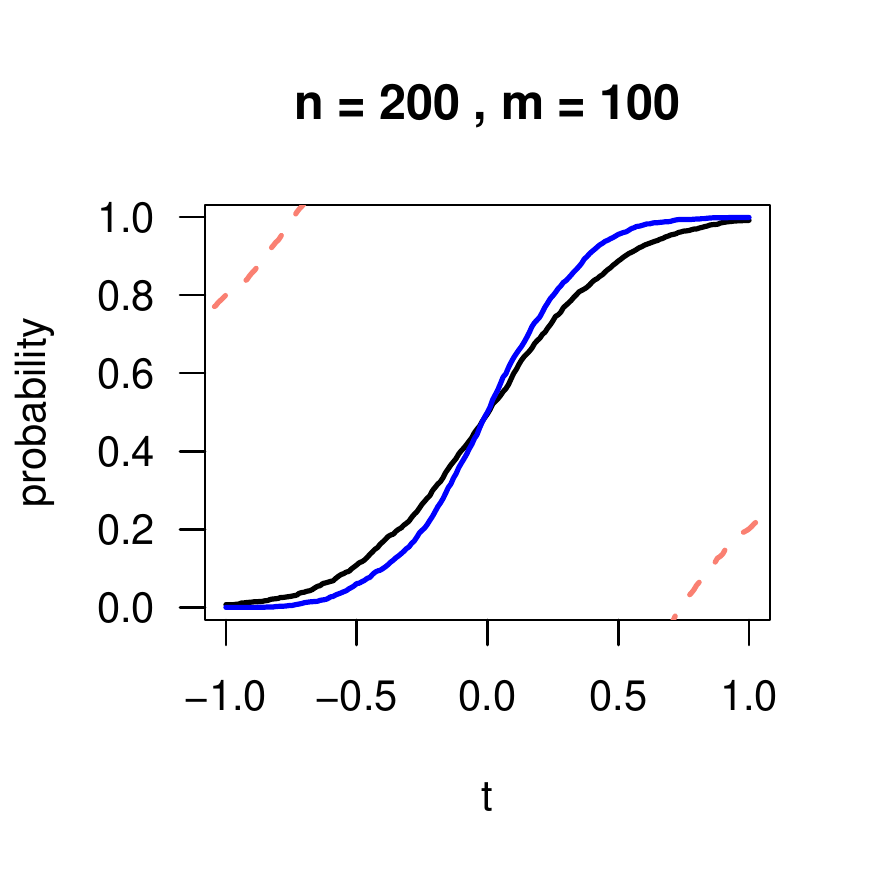}
  \includegraphics[width=0.31\textwidth]{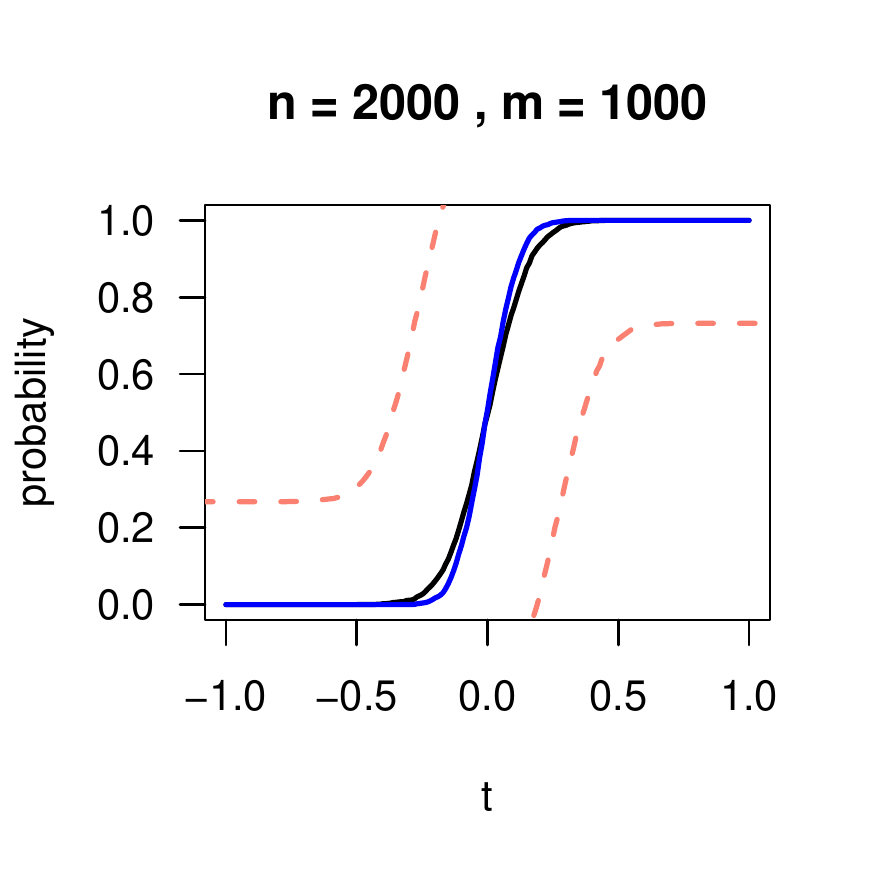}
  \includegraphics[width=0.31\textwidth]{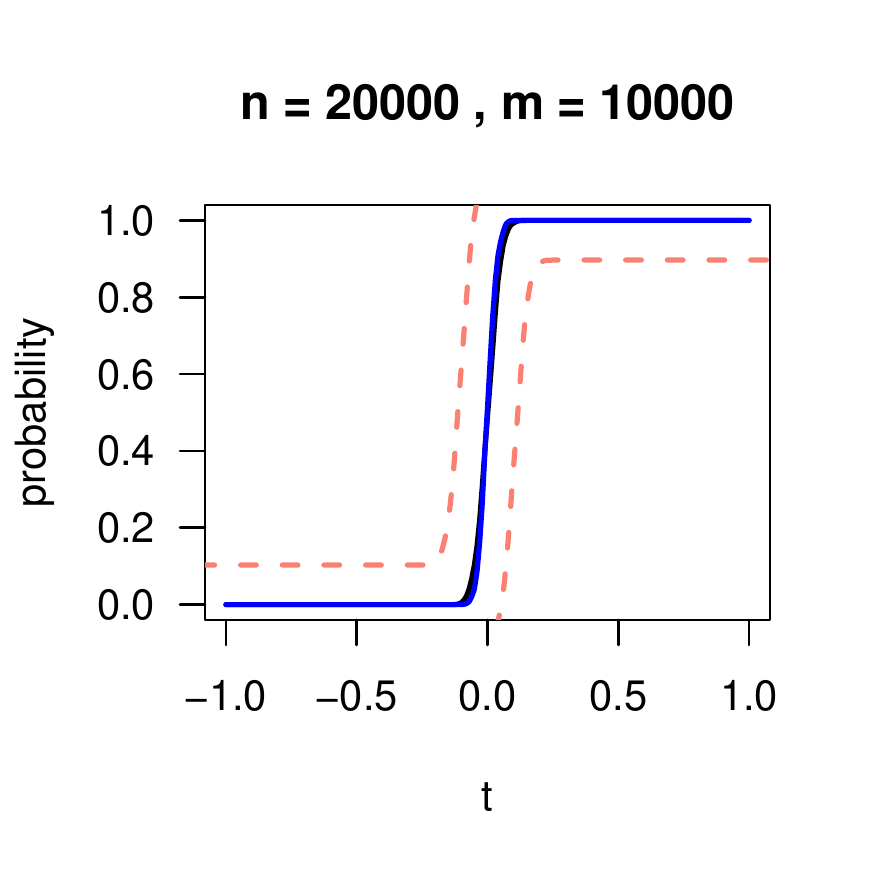}
  \caption{
    \label{fig:cdfComp}  
    A comparison of the cumulative distribution functions
    for $T(X)$ (black) and $T(\pi_g X)$ (blue)
    with the upper and lower bounds (dashed lines)
    from Theorem~\ref{thm:twoSample}.
  }
\end{figure}

\begin{figure}
    \centering
    \includegraphics[width=0.6\textwidth]{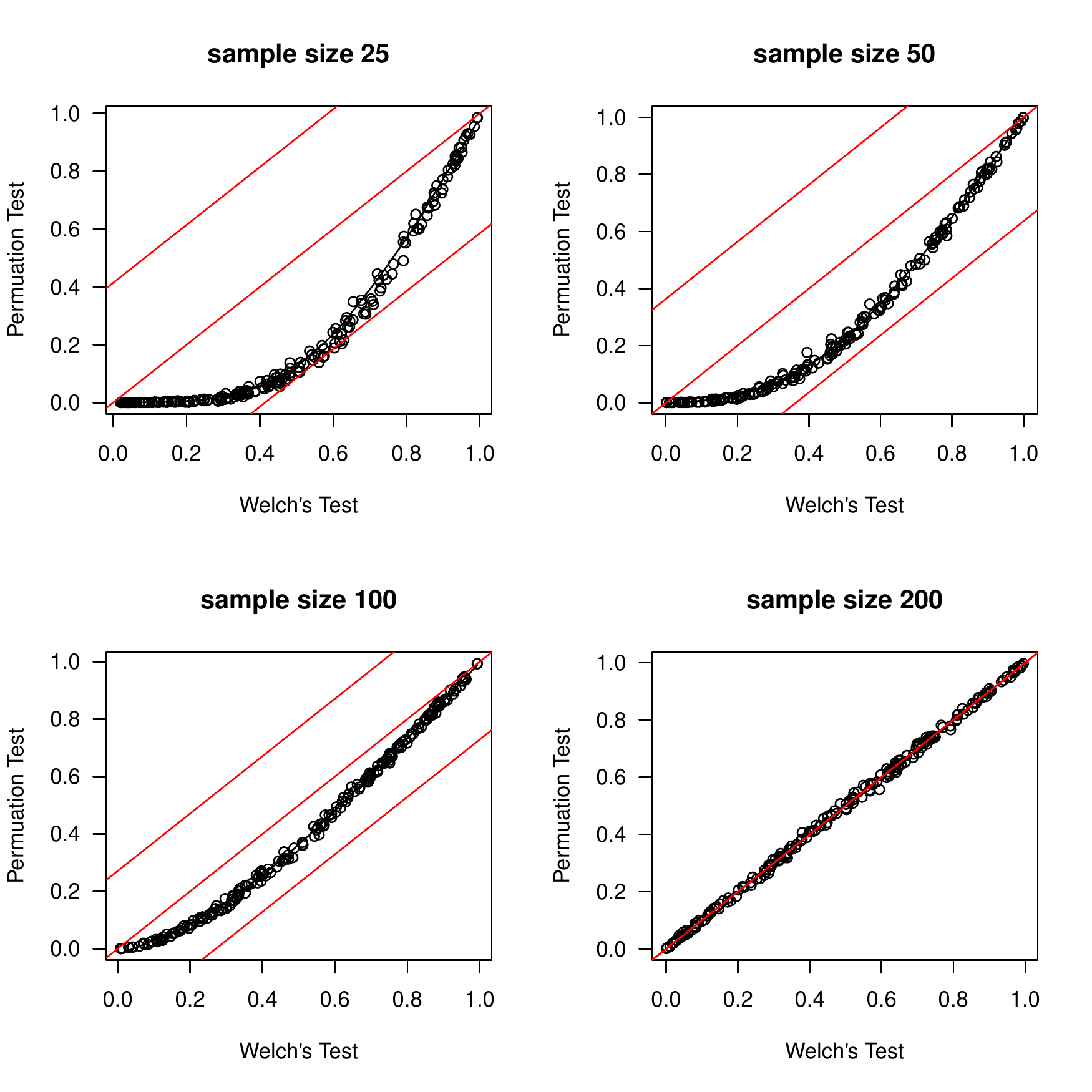}
    \caption{
      \label{fig:twoSamp}
      A comparison of the two measures in Theorem~\ref{thm:tvTwoGauss} 
      for $n=200$ and $m=25,50,100,200$.  
      The red lines correspond to the bounds on the total variation
      norm in Theorem~\ref{thm:tvTwoGauss}.
    }
\end{figure}

\section{Discussion}

It was demonstrated that 
nice functions of high dimensional
random variables are often nearly 
invariant to the actions
of a compact topological group.
This leads to interesting applications
such as random rotations of $\ell_p^n$-balls
allowing for the application of 
concentration inequalities for $SO(n)$
\citep{MECKES2019}.  The motivating
example of statistical hypothesis testing 
demonstrates that randomization tests
can be widely applicable even if the 
invariance assumption does not strictly
hold.

A deeper question in need of future investigation
is that of appropriate group selection.  
Many groups can satisfy 
the theorem conditions outlined in Section~\ref{sec:RandTest}.
However, each will induce a different collection of null 
hypothesis distributions.
As noted, some very recent work has considered
this question \citep{KONING_HEMERIK_2023,KONING2023}.
Secondly, it is not always possible to
parse the properties of a given statistical
test using direct methods.  The equivalence or 
asymptotic equivalence of randomization tests
offers a novel approach to understanding the 
strengths and limitations of a given statistical
test.

While Theorem~\ref{thm:asympInvariance} and 
Corollary~\ref{cor:asympSize} are certainly
useful, the quantitiative investigation in 
Section~\ref{sec:TwoTest} hints that future 
results concerning convergence in total variation
will also be of interest.  In particular, 
a comparison of the conditions required 
to prove such theorem with those presented in 
this work would give insight into different
types of convergence.

Lastly, a broader treatment of locally compact
topological groups is of future interest.
Of course, Haar measure would no longer be
a finite measure leading to problems concerning
how to think about a \textit{random} element
of a locally compact group.

\section*{Acknowledgements}

The author would like to acknowledge Professor Terry Gannon
at the University of Alberta for his helpful discussions on 
Group Theory and Representation Theory, 
Professor Sergii Myroshnychenko at Lakehead University
for recommending Elizabeth Meckes' text 
\textit{The random matrix theory of the classical compact groups}
\citep{MECKES2019},
and students Aneeljyot Alagh and Amitakshar Biswas
at the University of Alberta for 
enjoyable discussions on random rotations.

\appendix

\section{Direct Computation of $\ell_n^p$-balls}
\label{app:lpballs}

As noted in Section~\ref{sec:lpballs}, one can
directly compute the second moment of a uniform
point within an $\ell_n^p$-ball and thus 
derive convergence properties via a standard
application of Chebyshev's
inequality and the first Borel-Cantelli lemma.
Such tedious calculations for the $\ell_n^1$-ball 
and $\ell_n^2$-ball are included here for
completeness.  
Note that the volume of an $\ell^n_p$-ball of
radius $r>0$ is
$$
  \text{vol}_{n,p}(r) = 
  (2r)^n\frac{\Gamma(1+1/p)^n}{\Gamma(1+n/p)},
$$
which appears in a variety of sources such as
\cite{WANG2005,RABIEI2018}.
Thus, 
$\text{vol}_{n,1}(r) = (2r)^n/n!$
and
$\text{vol}_{n,2}(r) = r^n\pi^{n/2}/\Gamma(1+n/2)$.

\subsection{The $\ell^1_n$-ball}

Let $X\in\real^n$ be a uniformly random point
within the $\ell^1_n$-ball; that is, 
$\sum_{i=1}^n\abs{X_i}\le1$.
By symmetry, $\xv X_i = 0$ for all $i$ and
$\xv X_iX_j = 0$ for all $i\ne j$.  For the
second moment,
\begin{align*}
  \xv X_1^2 &=
  \frac{n!}{2^n}\int_{\ell_n^1} x_1^2 dx \\
  &= \frac{n!}{2^n}
     \int_{-1}^1 x_1^2\int_{\ell^1_{n-1}(1-\abs{x_1})}
     dx_{2,\ldots,n}dx_1\\ 
  &= \frac{n!}{2^n}\int_{-1}^1 x_1^2
     \frac{(2(1-\abs{x_1}))^{n-1}}{(n-1)!} dx_1 \\
  &= {n}\int_{0}^1 
     x_1^{3-1}(1-x_1)^{n-1} dx_1 \\
  &= {n}
     \frac{ \Gamma(3)\Gamma(n) }{\Gamma(n+3)}
   = \frac{2}{(n+1)(n+2)}.
\end{align*}
Thus, 
$\var{ \sum_{i=1}^n X_i } = 2n/(n+1)(n+2) \le 2/n$,
and by Chebyshev's inequality, for any $t>0$,
$$
  \prob{
    n^{-q}\abs{\sum_{i=1}^n X_i} \ge t
  } \le
  \frac{2}{n^{1+2q}}.
$$
As $\sum_{i=1}^\infty n^{-1-2q}<\infty$ for any
choice of $q>0$, we have that 
$$
  n^{-q}\abs{\sum_{i=1}^n X_i}\convas0
$$
as a consquence of the first Borel-Cantelli lemma.

\subsection{The $\ell^2_n$-ball}

A similar calculation can be performed for the 
$\ell^2$-ball.
Let $X\in\real^n$ be a uniformly random point
within the $\ell^2_n$-ball; that is, 
$\sum_{i=1}^nX_i^2\le1$.
By symmetry, $\xv X_i = 0$ for all $i$ and
$\xv X_iX_j = 0$ for all $i\ne j$.  For the
second moment,
\begin{align*}
  \xv X_1^2 &=
  \frac{\Gamma(1+n/2)}{\pi^{n/2}}
  \int_{\ell_n^2} x_1^2 dx \\
  &= \frac{\Gamma(1+n/2)}{\pi^{n/2}}
     \int_{-1}^1 x_1^2\int_{\ell^2_{n-1}(\sqrt{1-x_1^2})}
     dx_{2,\ldots,n}dx_1\\ 
  &= \frac{\Gamma(1+n/2)}{\pi^{n/2}}
     \int_{-1}^1 x_1^2
     \frac{(1-{x_1^2})^{(n-1)/2}
     \pi^{(n-1)/2}}{\Gamma(1+(n-1)/2)!} dx_1 \\
  &= \frac{2}{\sqrt{\pi}}
     \frac{\Gamma(1+n/2)}{\Gamma((n+1)/2)}
     \int_{0}^1 
     x_1^{2}(1-x_1^2)^{(n-1)/2} dx_1 \\
  &= \frac{1}{\sqrt{\pi}}
     \frac{\Gamma(1+n/2)}{\Gamma((n+1)/2)}
     \int_{0}^1 
     u^{1/2}(1-u)^{(n-1)/2} du \\
  &= \frac{1}{\sqrt{\pi}}
     \frac{\Gamma(1+n/2)}{\Gamma((n+1)/2)}
     \frac{\Gamma(3/2)\Gamma((n+1)/2)}{\Gamma(n/2+2)}
   = \frac{1}{n+2}
\end{align*}
Thus, 
$\var{ \sum_{i=1}^n X_i } = n/(n+2) \le 1$,
and by Chebyshev's inequality, for any $t>0$,
$$
  \prob{
    n^{-q}\abs{\sum_{i=1}^n X_i} \ge t
  } \le
  \frac{1}{n^{2q}}.
$$
As $\sum_{i=1}^\infty n^{-2q}<\infty$ for any
choice of $q>1/2$, we have that 
$$
  n^{-q}\abs{\sum_{i=1}^n X_i}\convas0
$$
as a consquence of the first Borel-Cantelli lemma.

\section{Simulating uniform points in an $\ell_p^n$ ball}
\label{app:ballSim}

For the sake of general interest and future simulations,
it is shown in this appendix how to generate
uniform random points within an $\ell_p^n$ ball using
the ratio of uniforms method
\citep[Section 3.7 and references therein]{FISHMAN2013}.
This is of interest, in particular, because the naive
acceptance-rejection approach of simulating 
$U = (U_1,\ldots,U_n)$ with entries 
\iid $\distUnifInt{-1}{1}$
and accepting if $\sum_{i=1}^nU_i^p \le 1$
will yield an acceptance with vanishingly small
probability for large $n$
as the volume of the $\ell^n_p$-ball is 
dwarfed by the volume of the hypercube.
That is,
$$
  \frac{\mathrm{vol}_{n,p}(1)}{2^n}
  = \frac{\Gamma(1+1/p)^n}{\Gamma(1+n/p)} \ll 1
$$
for large $n$.

Instead, let $Y = (Y_1,\ldots,Y_n)$ be a vector of 
\iid real valued random variables
with probability density proportional to 
$\exp(-\abs{t}^p)$, and let $Z\dist\distExp{1}$ be
independent of the $Y_i$'s.  Then, the vector
$$
  X = \frac{Y}{\left(
    \sum_{i=1}^n \abs{Y_i}^p + Z
  \right)^{1/p}}
$$
is a uniformly random point within the 
$\ell_p^n$-ball
\citep{BARTHE2005}.
Hence, the problem of generating such $X$'s is 
reduced to generating such $Y_i$'s.
For the special cases of $p=1,2,\infty$ specific
generation methods exist.  Otherwise, 
Algorithm~\ref{algo:ballSim} can be used for arbitrary
$p$.

The ratio of uniforms method is an acceptance-rejection
style algorithm where the ratio of 
two independent uniform random variables $U$ and $V$ 
is shown to have the desired distribution given 
an inequality is satisfied.  In this case,
let $r(t) = \exp(-\abs{t}^p)$, 
$U\dist\distUnifInt{0}{1}$ and 
$V\dist\mathrm{Uniform}[ 
            -\left({2}/{\ee p}\right)^{1/p} 
          , \left({2}/{\ee p}\right)^{1/p}]$.
Then, $Y = V/U$ has probability density proportional
to $\exp(-\abs{t}^p)$ if 
$$
  U^2 \le r( V/U ) = \exp( -\abs{V/U}^p )
$$
Some examples of this acceptance region are 
displayed in Figure~\ref{fig:ratUnif}.
The probability of such a random pair $(U,V)$
satisfying the above inequality can be quickly
calculated to be
$$
  \prob{ \text{Accept for }p } =
  \frac{\Gamma(1+1/p)}{2^{1+1/p}}({\ee p})^{1/p},
$$
which varies between 0.5 and 0.75 for values of $p\ge1$.

\begin{algorithm}[t]
	\caption{
		\label{algo:ballSim}
		Ratio of uniforms method of uniform points
        within the $\ell^n_p$-ball
	}
	\begin{tabbing}    
        \hspace{0.6\textwidth}\=\kill
		\qquad \enspace {\bf Initialize}
		$\eta = 0$ \> (Count how many variates are accepted)\\
		\qquad \enspace 
		{\bf While} $\eta < n$\\
		\qquad\qquad 
          Generate $U_1,\ldots,U_{n-\eta}\distiid
          \distUnifInt{0}{1}$\\
        \qquad\qquad Generate $V_1,\ldots,V_{n-\eta}\distiid
          \distUnifInt{ 
            -\left({2}/{\ee p}\right)^{1/p} 
          }{ \left({2}/{\ee p}\right)^{1/p}}$\\
        \qquad\qquad
          {\bf for each} $i\in\{1,\ldots,n-\eta\}$\\
        \qquad\qquad\qquad  
          {\bf if}\{ $U_i^2 \le \exp( -\abs{V_i/U_i}^p )$\}\\
        \qquad\qquad\qquad \enspace
          Keep the pair $(U_i,V_i)$\\
        \qquad\qquad\qquad \enspace
          $\eta \leftarrow \eta + 1$\\
        \qquad\qquad\qquad  
          {\bf else} \\
        \qquad\qquad\qquad \enspace
          Reject the pair $(U_i,V_i)$\\
        \qquad \enspace 
		   {\bf Return}
          $Y = ( V_1/U_1,\ldots,V_n/U_n )$
	\end{tabbing}
\end{algorithm}

\begin{figure}
  \centering
  \includegraphics[width=0.45\textwidth]{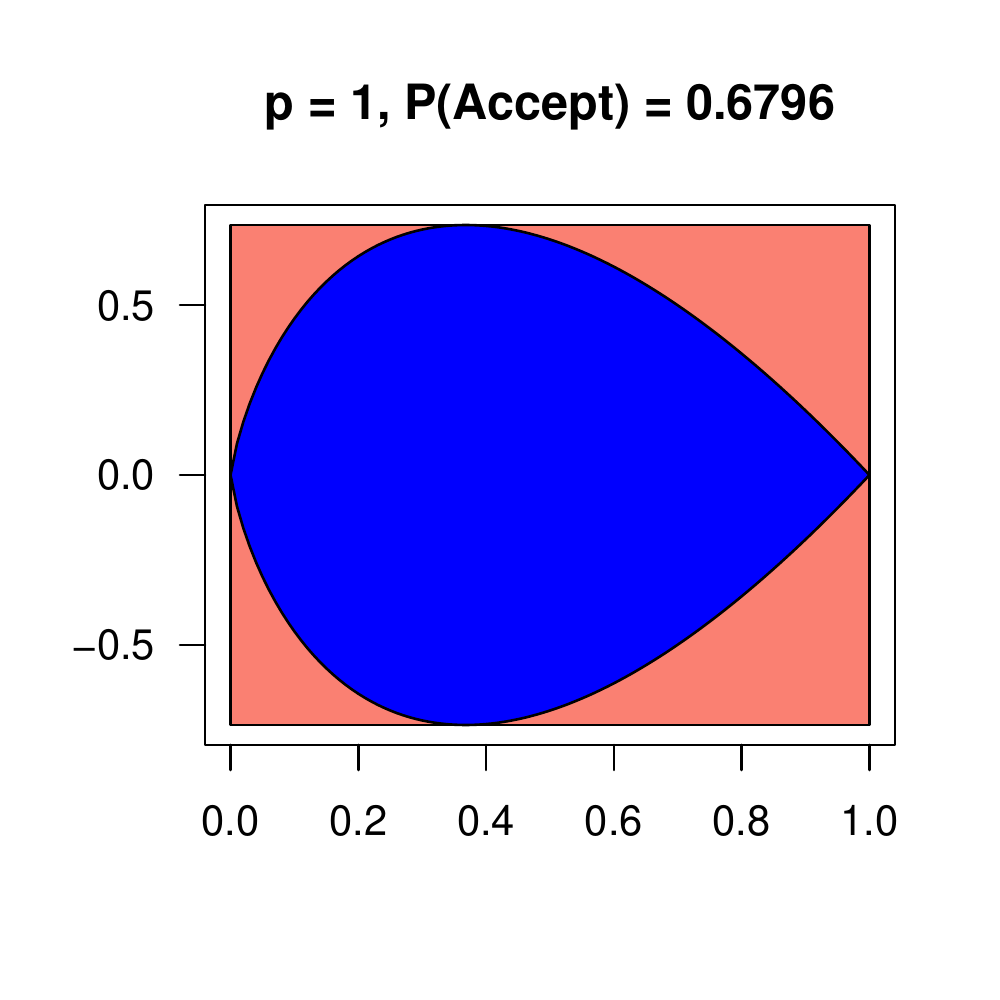}
  \includegraphics[width=0.45\textwidth]{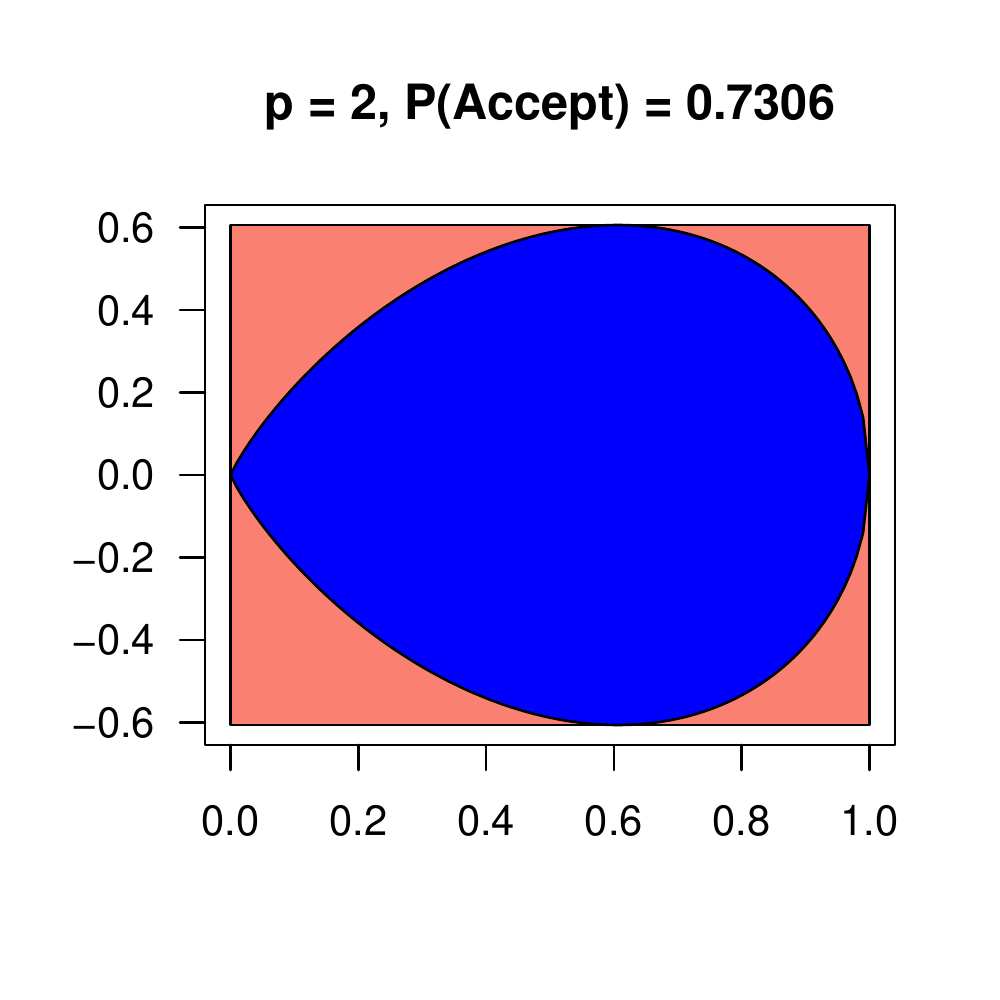}
  \includegraphics[width=0.45\textwidth]{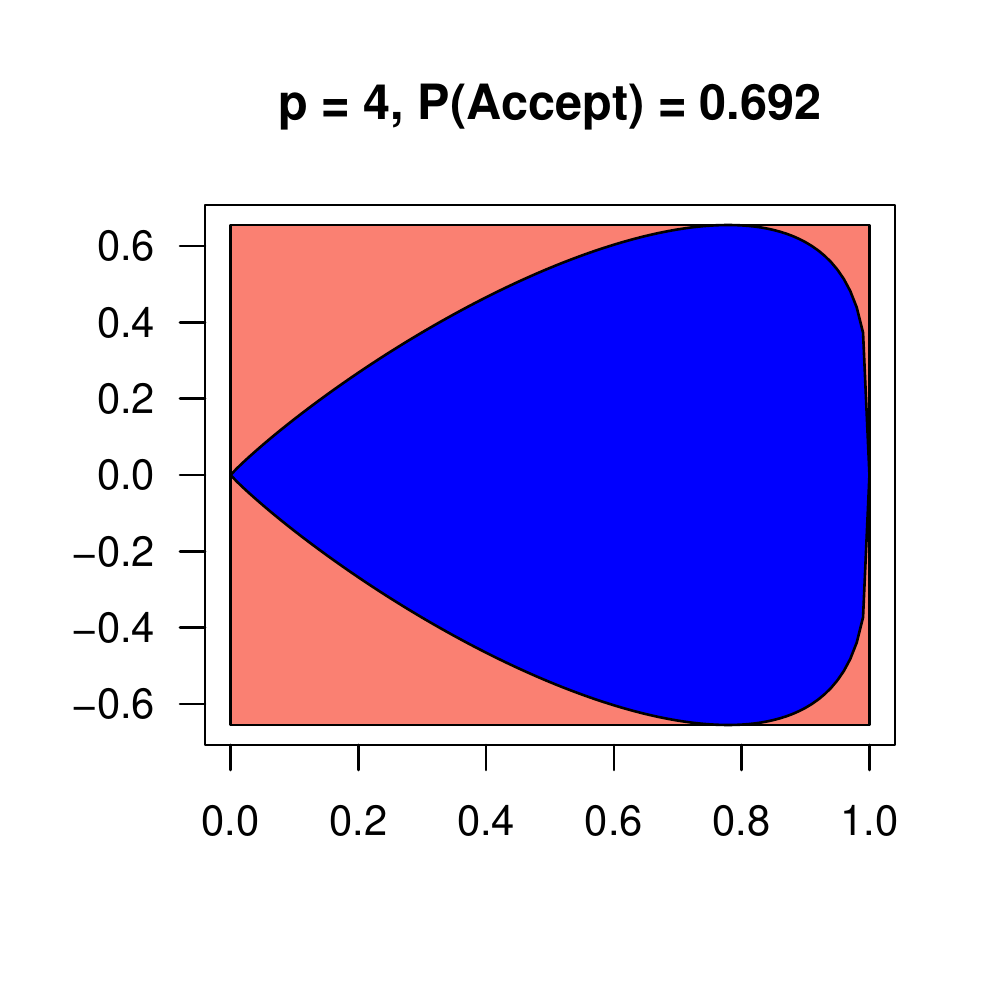}
  \includegraphics[width=0.45\textwidth]{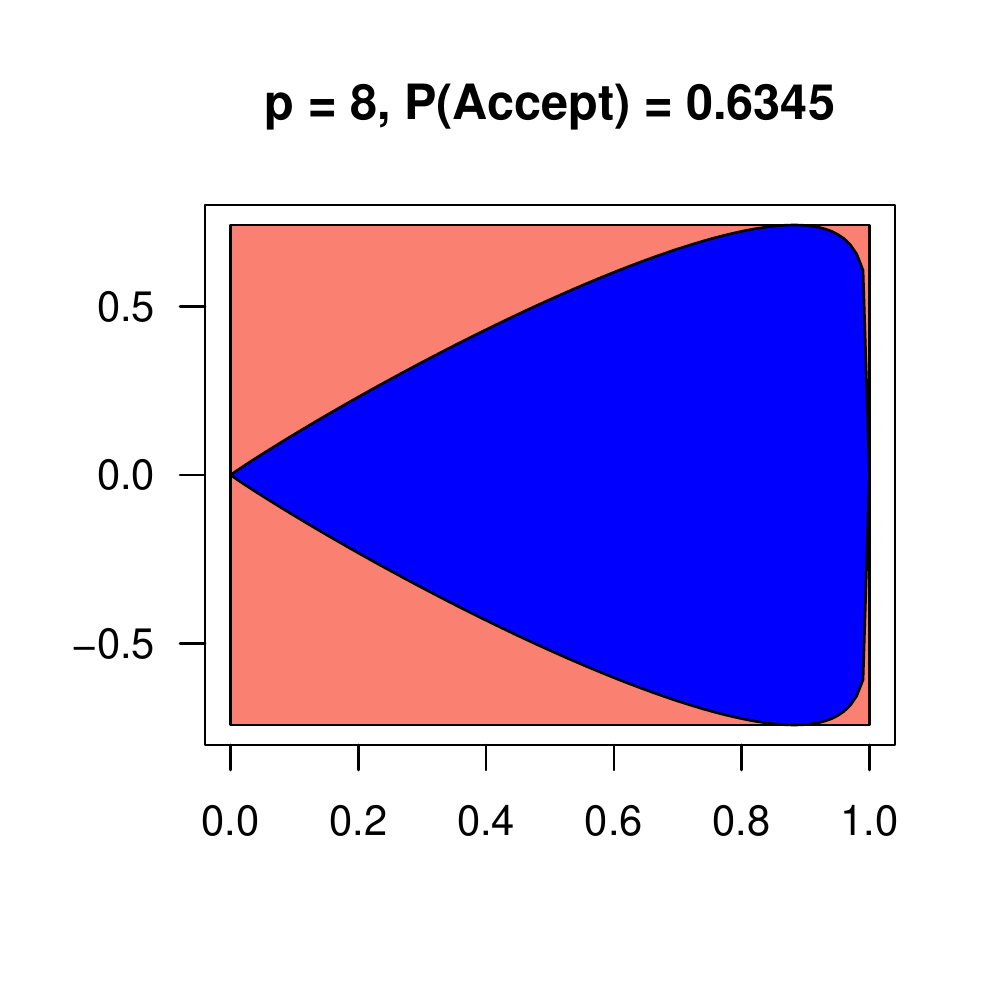}
  \caption{
    \label{fig:ratUnif}
    The acceptance region (blue) and probability for 
    the ratio of uniforms method applied for 
    $p=1,2,4,8$.
  }
\end{figure}

\bibliographystyle{plainnat}
\bibliography{kasharticle,kashbook,kashself,kashpack}

\end{document}